\documentclass[12pt]{l4dc2022}

\title[Distributed SNE Learning in Locally Coupled Network Games]{Distributed Stochastic Nash Equilibrium Learning in Locally Coupled Network Games with Unknown Parameters}
\usepackage{times}
\usepackage{mathrsfs}
\usepackage{bbm}
\usepackage{amsmath, amsfonts, outlines, mathtools, outlines, cancel, hyperref, stackengine, bbm, enumitem}
\usepackage{algorithm2e}
\usepackage{multibib}

\newcites{Apndx}{References in Appendix}

\setenumerate[1]{label=(\roman*)}

\newtheorem{assumption}{Assumption}

\author{%
 \Name{Yuanhanqing Huang} \Email{huan1282@purdue.edu}\\
 \addr Elmore Family School of Electrical and Computer Engineering, Purdue University, USA
 \AND
 \Name{Jianghai Hu} \Email{jianghai@purdue.edu}\\
 \addr Elmore Family School of Electrical and Computer Engineering, Purdue University, USA
}

\newcommand{\norm}[1]{\lVert#1\rVert}

\newcommand{\zer}[1]{\text{Zer}(#1)}
\newcommand{\fix}[1]{\text{Fix}(#1)}
\newcommand{\blkd}[1]{\text{blkd}(#1)}

\newcommand{\expt}[2]{\mathbb{E}_{#1}[#2]}
\newcommand{\bexpt}[2]{\mathbb{E}_{#1}\big[#2\big]}
\DeclarePairedDelimiter\ceil{\lceil}{\rceil} 

\newcommand\ubar[1]{\stackunder[1.2pt]{#1}{\rule{.8ex}{.075ex}}}

\newcommand{\gra}[1]{\text{gra}(#1)}
\DeclareMathOperator{\dom}{dom}
\newcommand{\abs}[1]{\lvert #1 \rvert}

\newcommand{\RN}[1]{%
  \textup{\uppercase\expandafter{\romannumeral#1}}%
}

\usepackage{stackengine}
\stackMath
\newcommand\tsup[2][2]{%
 \def\useanchorwidth{T}%
  \ifnum#1>1%
    \stackon[-.5pt]{\tsup[\numexpr#1-1\relax]{#2}}{\scriptscriptstyle\sim}%
  \else%
    \stackon[.5pt]{#2}{\scriptscriptstyle\sim}%
  \fi%
}

\newcommand{\supsub}[3]{{#1}^{\scriptscriptstyle #2}_{\scriptscriptstyle #3}}

\DeclareMathOperator{\minimize}{minimize}

\DeclareMathOperator{\subj}{subject\:to}

\newcommand{\playerN}{\mathcal{N}}
\newcommand{\edgeE}{\mathcal{E}}

\newcommand{\by}{\boldsymbol y}

\newcommand{\bone}{\boldsymbol 1}
\newcommand{\bzero}{\boldsymbol 0}
\newcommand{\btau}{\boldsymbol \tau}

\newcommand{\rset}[2]{\mathbb{R}^{#1}_{#2}}
\newcommand{\pdset}[2]{\mathbb{S}^{#1}_{#2}}
\newcommand{\crset}[2]{\overline{\mathbb{R}}^{#1}_{#2}}
\newcommand{\nset}[2]{\mathbb{N}^{#1}_{#2}}

\newcommand{\neighbN}[2]{\mathcal{N}^{#1}_{#2}}

\DeclareMathOperator{\argmin}{argmin}

\newcommand{\blambda}{\boldsymbol \lambda}

\newcommand{\diagA}{\Lambda}

\newcommand{\pspace}{\mathcal{K}}
\DeclareMathOperator{\proj}{Pj}
\newcommand{\res}[1]{\text{res}(#1)}

\DeclareMathOperator{\optT}{\mathbb{T}}
\DeclareMathOperator{\optA}{\mathbb{A}}
\DeclareMathOperator{\optB}{\mathcal{B}}

\DeclareMathOperator{\gjacob}{\mathbb{F}}
\DeclareMathOperator{\extgjacob}{\tilde{\mathbb{F}}}

\begin{document}

\maketitle

\begin{abstract}%

In stochastic Nash equilibrium problems (SNEPs), it is natural for players to be uncertain about their complex environments and have multi-dimensional unknown parameters in their models. 
Among various SNEPs, this paper focuses on locally coupled network games where the objective of each rational player is subject to the aggregate influence of its neighbors. 
We propose a distributed learning algorithm based on the proximal-point iteration and ordinary least-square estimator, where each player repeatedly updates the local estimates of neighboring decisions, makes its augmented best-response decisions given the current estimated parameters, receives the realized objective values, and learns the unknown parameters. 
Leveraging the Robbins-Siegmund theorem and the law of large deviations for M-estimators, we establish the almost sure convergence of the proposed algorithm to solutions of SNEPs when the updating step sizes decay at a proper rate. 
\end{abstract}

\thispagestyle{plain} 

\section{Introduction}

Nash equilibrium problems, rooted in the seminal work by \cite{nash1950equilibrium}, model and describe the interactions among multiple decision-makers or players where they aim at optimizing their own payoffs given the strategies of others. 
In a stochastic Nash equilibrium problem (SNEP), players take the uncertainty in their payoffs into account when deciding on their actions (\cite{facchinei2010generalized, shanbhag2006decomposition}). 
This type of problem can be applied to model a considerable number of applications such as power markets (\cite{kannan2011strategic,kannan2013addressing}), engagement of multiple humanitarian organizations in disaster relief (\cite{nagurney2020stochastic}), and the traffic assignment of strategic risk-averse users (\cite{nikolova2014mean}), to name a few. 

The past decade has witnessed significant progress in the distributed solution of Nash equilibrium problems (NEPs) under both deterministic and stochastic setups (\cite{yi2019operator, pavel2019distributed, bianchi2020fast, shi2017lana}). 
It is often assumed that either each player is allowed to communicate with the players affecting its objective or each player maintains a local estimate of the decisions of all the other players. 
To model network games with better scalability, considerable effort has been spent on studying network games with special structures, such as average aggregative games (AAGs) and network aggregative games (NAGs) (\cite{parise2021analysis}). 
Here, we focus on locally coupled network games, where the objective of each player depends on its own decision and some linear transformation of the decisions of its neighbors, determined by an underlying communication network. 
In particular, if the influence of its neighbors can be expressed as a convex combination of their decisions, it corresponds to NAGs, the properties and distributed solutions of which have been extensively investigated in \cite{parise2015network,parise2020distributed}.

Most of the aforementioned work designs distributed solutions under the fundamental assumption that each player has perfect knowledge of the payoff function. 
Nevertheless, in general, players may not be perfectly aware of their environments, and the outcomes of their actions may not always coincide with the predictions (\cite{kirman1975learning, frydman1982towards}). 
Consequently, each player needs to modify its model or the parameters of its models in light of the observations it makes (\cite{esponda2021asymptotic}). 
On the other end of the spectrum, there is an emerging research interest in finding Nash equilibria through bandit/zeroth-order online learning schemes, assuming the players are completely oblivious to the game mechanism, perhaps even ignoring its existence (\cite{bravo2018bandit, heliou2021zeroth,tatarenko2020bandit}, etc.). 
In this work, we consider the setting where the players are aware of their own objectives' functional forms and feasibility constraints while uncertain about some parameters in their objectives, and they participate in sequential repetitions of the same game while learning the parameters over this process. 

To learn equilibria while confronted with unknown parameters, the authors of \cite{jiang2017distributed} present two distributed schemes to solve SNEP without observations of others' strategies. 
The first scheme is based on the stochastic gradient method which constructs the learning problems independent of the computation of Nash equilibria (NEs); the second scheme solves stochastic Nash-Cournot games via iterative fixed-point methods under a common knowledge assumption concerning the cost functions and strategy sets of their competitors. 
The authors of \cite{lei2020asynchronous} further extend the first scheme in \cite{jiang2017distributed} and design an asynchronous inexact proximal best-response solution to the unknown problems. 
Nevertheless, in most practical applications and online learning settings, the parameter learning and NE seeking processes tangle with each other and players are unwilling to share their local information over the whole network. 
On that account, the authors of \cite{meigs2017learning, meigs2019learning} instead consider the case where the parameter estimation process is intrinsically coupled with the strategy update process. 
Moreover, they postulate that each player can observe the necessary information for parameter learning (e.g. aggregates of neighbors' strategies in NAGs) without the common-knowledge assumption. 
The learning algorithm considered in this paper is similar to the one in \cite{meigs2019learning}, while we extend the results by proposing a solution that can handle a more general class of games with multi-dimensional unknown parameters in objectives.
Furthermore, we establish the convergence without requiring the contractiveness of the NE-seeking algorithm, and the theoretical analysis can be further extended to the solutions of generalized Nash equilibrium problems (GNEPs) \cite{facchinei2010generalized}. 

In this paper, we develop a distributed learning algorithm that guarantees almost-sure convergence to stochastic Nash equilibria (SNEs) in locally coupled network games with unknown parameters. 
We assume that after all players determine their decisions, each player can observe its own realized objective value and the decisions made by its neighbors.
At each iteration, every player selects its decision indicated by the solution of its augmented best-response function parameterized by the current parameter estimates along with some random exploration vector. 
Then each player receives feedback about the objective values and neighbors' decisions and updates its parameters via an ordinary least squares estimator (OLSE). 
Furthermore, unlike most of the existing work that enjoys contractive iterations in the NE seeking dynamics, the fixed-point iteration operator considered in this work only satisfies (quasi)nonexpansiveness, due to the partial-information setting (and the global resource constraints in GNEPs). 
By leveraging the Robbins-Siegmund theorem, we establish the main convergence theorem for the proposed algorithm and discuss the conditions needed to ensure the convergence to solutions. 
We derive an upper bound for the asymptotic convergence rate of the OLSE and discuss the proper choice of step sizes to guarantee the convergence. 
The technical proofs and complementary examples and discussions are included in \cite{huang2021distlearn}

\textit{Basic Notations:} For a set of matrices $\{V_i\}_{i \in S}$, we let $\blkd{V_1, \ldots, V_{|S|}}$ or $\blkd{V_i}_{i \in S}$ denote the diagonal concatenation of these matrices, $[V_1, \ldots, V_{|S|}]$ their horizontal stack, and $[V_1; \cdots; V_{|S|}]$ their vertical stack. 
For a set of vectors $\{v_i\}_{i \in S}$, $[v_i]_{i \in S}$ or $[v_1; \cdots; v_{|S|}]$ denotes their vertical stack. 
For a matrix $V$ and a pair of positive integers $(i, j)$, $[V]_{(i,j)}$ denotes the $i$th row and the $j$th column of $V$. 
For a vector $v$ and a positive integer $i$, $[v]_i$ denotes the $i$th entry of $v$. 
Denote $\crset{}{} \coloneqq \rset{}{} \cup \{+\infty\}$, $\rset{}{+} \coloneqq [0, +\infty)$, and $\rset{}{++} \coloneqq (0, +\infty)$. 
$\pdset{n}{+}$ (resp. $S^n_{++}$) represents the set of all $n\times n$ symmetric positive semi-definite (resp. definite) matrices.
$\iota_{\mathcal{S}}(x)$ is defined to be the indicator function of a set $\mathcal{S}$, i.e., if $x \in \mathcal{S}$, then $\iota_{\mathcal{S}}(x) = 0$; otherwise, $\iota_{\mathcal{S}}(x) = +\infty$. 
$N_{S}(x)$ denotes the normal cone to the set $S \subseteq \rset{n}{}$ at the point $x$: if $x \in S$, then $N_S(x) \coloneqq \{u \in \rset{n}{} \mid \sup_{z \in S} \langle u, z-x \rangle \leq 0 \}$; otherwise, $N_S(x) \coloneqq \varnothing$. 
If $S \in \rset{n}{}$ is a closed and convex set, the map $\proj_S:\rset{n}{} \to S$ denotes the projection onto $S$, i.e., $\proj_S(x) \coloneqq \argmin_{v \in S} \norm{v - x}_2$. 
We use $\rightrightarrows$ to indicate a point-to-set map. 
For an operator $T: \rset{n}{} \rightrightarrows \rset{n}{}$, $\zer{T} \coloneqq \{x \in \rset{n}{} \mid Tx \ni \bzero\}$ and $\fix{T} \coloneqq \{x \in \rset{n}{} \mid Tx \ni x\}$ denote its zero set and fixed point set, respectively. 
We denote $\dom(T)$ the domain of the operator $T$ and $\gra{T}$ the graph of it.
The resolvent and reflected resolvent of $T$ are defined as $J_T \coloneqq (I + T)^{-1}$ and $R_T \coloneqq 2J_T - I$, respectively.

\section{Preliminaries of Stochastic Locally Coupled Network Games}

\subsection{Formulation of Locally Coupled Network Games}

We consider a game played among a group of self-interested players indexed by $\playerN \coloneqq \{1, \ldots, N\}$, whose interactions are specified by an underlying communication network $\mathcal{G} = (\mathcal{N}, \mathcal{E})$. 
We use $(i, j)$ to denote a directed edge having player $i$ as its tail and $j$ as its head. 
Although each edge $e\in \mathcal{E}$ admits certain direction, the communication through the edge $e$ are undirected, i.e., each player $i$ can send messages to both its in-neighbors $\neighbN{+}{i} \coloneqq \{j \in \playerN \mid (j,i) \in \mathcal{E}\}$ and out-neighbors $\neighbN{-}{i} \coloneqq \{j \in \playerN \mid (i,j) \in \mathcal{E}\}$, the cardinalities of which are denoted by $N^+_i$ and $N^-_i$, respectively. 

\begin{assumption}{(Communicability)}\label{asp:commtopo}
The underlying communication graph $\mathcal{G} = (\mathcal{N}, \edgeE)$ is undirected and connected. Furthermore, it has no self-loops. 
\end{assumption}

The goal of each player $i$ is to minimize an expected-value objective defined by $\mathbb{J}_i(x_i; x^+_i, w^*_i) \coloneqq \expt{}{J_i(x_i; s_i(x^+_i; \xi_i, w^*_i))}$ which depends on its own decision $x_i \in \rset{n_i}{}$ and the decisions of its in-neighbors $x^+_i \coloneqq [x_{j}]_{j \in \neighbN{+}{i}}$. 
It is worth mentioning that $x^+_i$ are treated as parametric inputs of $\mathbb{J}_i$.
We use $s_i$ to denote the neighboring aggregate function $s_i(x^+_i; \xi_i, w^*_i) \coloneqq w^*_{ii} + \sum_{j \in \neighbN{+}{i}}w^{*T}_{ji}x_j + \xi_i$, where $w^*_{ii} \in \rset{}{}$ and $w^*_{ji} \in \rset{n_j}{}$ are some constant parameters, and we let the random variable $\xi_i: \Omega \to \rset{}{}$ capture uncertainty in $s_i$. 
The local decision $x_i$ made by player $i$ is subject to a set of local feasibility constraints $\mathcal{X}_i \subseteq \rset{n_i}{}$. 
We further define $w^*_i \coloneqq [w^*_{ji}]_{j \in \{i\} \cup \neighbN{+}{i}}$, $w^* \coloneqq [w^*_i]_{i \in \playerN}$, $n\coloneqq \sum_{i \in \playerN} n_i$, ${n}^+_{i} = \sum_{j \in \neighbN{+}{i}}n_j$, and $\mathcal{X} \coloneqq \prod_{i \in \playerN} \mathcal{X}_i$. 
The feasible parameter set of player $i$ is denoted by $\mathcal{W}_i$, and $w^*_i \in \mathcal{W}_i$. 
Altogether, the local stochastic optimization problem of player $i$ can be formally written as:
\begin{equation}\label{eq:pf-optprob}
\minimize_{x_i \in \mathcal{X}_i} \mathbb{J}_i(x_i; x^+_i, w^*_i)
\end{equation}

The solution concept of the problem described in \eqref{eq:pf-optprob} we focus on in this paper is stochastic Nash equilibria (SNEs) (\cite{ravat2011characterization}), whose definition is given as follows:
\begin{definition}
The collective decision $x^* \in \mathcal{X}$ is an SNE if no player can benefit by unilaterally deviating from $x^*$, i.e., $\forall i \in \playerN$, $\mathbb{J}_i(x^*_i; x^{*+}_{i}, w^*_i) \leq \mathbb{J}_i(x_i; x^{*+}_{i}, w^*_i)$ for any $x_i \in \mathcal{X}_i$. 
\end{definition}

We then make the following regularity assumptions concerning the objective functions, feasible sets, and solution sets.
In particular, Assumption~\ref{asp:obj} (iii) is imposed to facilitate our later analysis regarding the convergence of the parameter learning and SNE seeking algorithm. 
\begin{assumption}\label{asp:obj} (Local Objectives)
For each $i \in \playerN$, given any fixed sample $\omega_i \in \Omega_i$ and the precise parameters $w^*_i$, the scenario-based and expected-value objectives $J_i$ and $\mathbb{J}_i$ satisfy: 

(i) $J_i(x_i; s_i(x^{+}_{i}, \xi_i; w^*_i))$ is convex in $x_i$ given any fixed $x^{+}_{i}$; 

(ii) $J_i(x_i; s_i(x^{+}_{i}, \xi_i; w^*_i))$ is proper and lower-semicontinuous in $x_i$ and $x^{+}_{i}$; 

(iii) $\mathbb{J}_i$ can be written as $\mathbb{J}_i(x_i; x^{+}_{i}, \hat{w}_i) = \mathbbm{f}_i(x_i; x^{+}_{i}) + \mathbbm{g}_i(x_i; x^{+}_{i}, \hat{w}_i)$, where for any fixed $x^{+}_{i}$, $\mathbbm{f}_i$ is convex in $x_i$ and $\mathbbm{g}_i$ is differentiable in $x_i$. Moreover, $\mathbbm{g}_i$ is continuous in $\hat{w}_i$ and $\nabla_{x_i}\mathbbm{g}_i$ is Lipschitz in $\hat{w}_i$ with the constant $\alpha_{g,i}\norm{x^{+}_{i}}_2 + \beta_{g,i}$ ($\alpha_{g,i}, \beta_{g,i} \geq 0$) on $\mathcal{W}_i$ for any fixed $x_i \in \mathcal{X}_i$ and $x^{+}_{i}$. 
\end{assumption}
\begin{assumption}\label{asp:feaset} (Feasible Sets)
For each $i \in \playerN$, $\mathcal{X}_i$ is nonempty, compact, convex, and satisfies Slater's constraint qualification (CQ). 
\end{assumption}
\begin{assumption}\label{asp:exist-sol} (Existence of SNE)
The SNEP considered admits a nonempty set of SNEs.
\end{assumption}

By stacking the partial gradients $\partial_{x_i}\mathbb{J}_i(x_i; x^{+}_{i}, w^*_i)$, we can construct the so-called pseudo-gradient operator $\gjacob_{w^*}: \mathcal{X} \rightrightarrows \rset{n}{}$ as 
$\gjacob_{w^*}: x \mapsto [\partial_{x_i}\mathbb{J}_i(x_i; x^{+}_{i}, w^*_i)]_{i \in \playerN}$. 
This operator plays a significant role in regulating different types of games, analyzing the properties of solution sets, etc. 
Games with maximally monotone pseudo-gradient $\gjacob_{w^*}$ are called monotone games. 
As has been shown in \cite[Prop.~12.4, Sect.~12.2.3]{palomar2010convex}, to compute SNEs of \eqref{eq:pf-optprob}, we can instead solve the corresponding generalized variational inequality (GVI): 
find a pair of vectors $(x^*, g^*)$ such that $x^* \in \mathcal{X}$ and $g^* \in \gjacob_{w^*}(x^*)$ and $(x - x^*)^Tg^* \geq 0, \forall x \in \mathcal{X}.$
The Karush-Kuhn-Tucker (KKT) problem of the corresponding GVI can be written as follows:
\begin{equation}\label{eq:kkt-gvi}
\bzero \in \partial_{x_i}\mathbb{J}_i(x_i; x^{+}_{i}, w^*_i) + N_{\mathcal{X}_i}(x_i), \forall i \in \playerN, 
\end{equation}
under the proper CQ. 
To motivate our analysis, we briefly discuss one typical example of locally coupled network games below and another one in \cite[Appendix~G]{huang2021distributed}. 
\begin{example}(Scalar linear quadratic games (\cite{parise2019variational}))
There is a finite set of players indexed by $i = 1, \ldots, N$, each making a scalar non-negative bounded strategy $x_i$ to optimizing its quadratic objective $J_i(x_i; x^+_i) \coloneqq \frac{1}{2}(x_i)^2 + (K_i \sum_{j \in \neighbN{+}{i}}w_{ij}x_j - a_j)x_i$, where $K_i, a_i \in \rset{}{}$, and $w_{ij}$ indicates the influence of player $j$'s decision on the objective function of player $i$. 
This model has been applied to investigate various economic settings including the private provision of public goods and games with local payoff complementarities but global substitutability. 
For more examples of locally coupled network games satisfying the assumptions discussed, see \cite{parise2019variational, parise2021analysis} and the references therein. 
\end{example}

\subsection{Distributed Solution via Proximal-Point Algorithm with Precise Parameters}\label{subsect:dist-sol-ppa}

We start by proposing a distributed solution for the SNE problem with precise knowledge of the involved parameters. 
Some iterative algorithms for computing NEs with an emphasis on algorithms amenable to decomposition have been proposed, such as the best-response (BR) iteration \cite{meigs2019learning} and the proximal BR iteration \cite[Sec.~12.6]{palomar2010convex}. 
In these algorithms, the implementation of the fixed-point scheme can be carried out in a distributed and independent manner, and each player makes decisions based on others' decisions from the last iteration. 
Nevertheless, the convergence properties rely on the contractiveness of their fixed-point iterations, which requires additional regularities on the objectives and network structures.  

To tackle a more general class of games, we design the fixed-point iteration in light of the proximal-point algorithm (PPA) and the Krasnosel'skii-Mann algorithm (KM) \cite[Thm.~23.41, Thm.~5.15]{BauschkeHeinzH2017CAaM}. 
Given a finite-dimensional maximally monotone operator $A$, the resolvent $J_{A} \coloneqq (I + A)^{-1}$ is firmly nonexpansive. 
Combining PPA and KM yields the following fixed-point iteration prototype 
$x^{(k+1)} \coloneqq x^{(k)} + \gamma^{(k)}(J_{A}x^{(k)} - x^{(k)})$, 
which will generate a sequence converging to a point in $\zer{A}$, the zero set of $A$. 
Here, $(\gamma^{(k)})_{k \in \nset{}{}}$ satisfies $\gamma^{(k)} \in [0, 1]$ and $\sum_{n \in \nset{}{}} \gamma^{(k)}(1 - \gamma^{(k)}) = +\infty$. 

For each player $i \in \playerN$, we endow it with a local estimate $y^j_i$ for the decision of each of its in-neighbors $j \in \neighbN{+}{i}$. 
In what follows, we use $y^i_i$ to denote the local decision of player $i$, ${y}^+_i \coloneqq [y^j_i]_{j \in \neighbN{+}{i}}$ the stack of the local estimates of its in-neighbors' decisions, and $y^{-}_{i} \coloneqq [y^i_j]_{j \in \neighbN{-}{i}}$ the stack of local estimates of $y^i_i$ maintained by player $i$'s out-neighbors. 
Let $y_i \coloneqq [y^i_i; [y^j_i]_{j \in \neighbN{+}{i}}]$, and $y \coloneqq [y_i]_{i \in \playerN}$. 
The feasible region of the stack vector $y_i$ is given by $\tilde{\mathcal{X}}_i \coloneqq \mathcal{X}_i \times \rset{n^+_i}{}$. 
Let $\tilde{\mathcal{X}} \coloneqq \prod_{i \in \playerN} \tilde{\mathcal{X}}_i \subseteq \rset{\tilde{n}}{} $, where $\tilde{n} \coloneqq n + \sum_{i \in \playerN} n^+_i$. 
With the introduction of local estimates, we can construct the extended pseudo-gradient $\extgjacob_{w^*}: \tilde{\mathcal{X}} \rightrightarrows \rset{n}{}$ as $\extgjacob_{w^*}: y \mapsto [\partial_{y^i_i}\mathbb{J}_i(y^i_i; y^+_i, w^*_i)]_{i \in \playerN}$, and the selection matrix $\mathcal{R}$ as $\mathcal{R} \coloneqq \blkd{\{\mathcal{R}_i\}_{i \in \playerN}}$, where each $\mathcal{R}_i \coloneqq [I_{n_i}, \bzero_{n_i \times n^+_i}]$. 

The use of local estimate $y^j_i$ can be interpreted as introducing a "pseudo-player" $j_i$ into the network. 
Each pseudo-player $j_i$ is connected to player $j$, where player $j$ and pseudo-player $j_i$ for all $i \in \neighbN{-}{j}$ constitute a connected component. 
We then conceptually disconnect the edges in $\edgeE$, which gives rise to a new dependency network $\tilde{\mathcal{G}}$ with $N$ such connected components. 
Let $L$ denote the Laplace matrix of the network $\tilde{\mathcal{G}}$, which has eigenvalue zero with multiplicity $N$ and other eigenvalues greater than zero. 
Each zero eigenvalue is associated with an eigenvector corresponding to the consensus within a connected component. 
By further extending each entry of $L$ to a square matrix with proper dimension, we obtain a square matrix $\tilde{L} \in \rset{\tilde{n} \times \tilde{n}}{}$ that we can leverage to obtain a compact form of the fixed-point iteration. 
We define the following operator $\optT$ whose zeros correspond to exactly the SNEs of \eqref{eq:pf-optprob}: 
\begin{align}
\begin{split}
\optT: 
y \mapsto \partial ( \sum_{i \in \playerN} \mathbb{J}_i(y^i_i; {y}^+_i, w^*_i) + \iota_{\mathcal{X}_i}(y^i_i)) 
 + \rho \tilde{L}y = \mathcal{R}^T\extgjacob_{w^*}(y) + N_{\tilde{\mathcal{X}}}(y) + \rho \tilde{L}y,
\end{split}
\end{align}
where $\rho \in \rset{}{++}$ is a constant controlling the contribution of local estimation errors. 
The formal statement of the equivalence is given below and the proof is reported in \cite[Appendix~A]{huang2021distlearn}. 
\begin{theorem}\label{thm:eq-prob}
Suppose Assumptions~\ref{asp:commtopo} to \ref{asp:feaset} hold, and there exists $y^* \in \zer{\optT}$. Then $\tilde{L}y^* = \bzero$ and the tuple $\{y^{i*}_i\}_{i \in \playerN}$ satisfies the KKT conditions \eqref{eq:kkt-gvi} for an SNE. 
Conversely, if the problem \eqref{eq:pf-optprob} has a solution $\{y^{i\dagger}_i\}_{i \in \playerN}$, then there exist local estimates $y^{j\dagger}_i$ such that their stack $y^\dagger \in \zer{\optT}$. 
\end{theorem}

Since the matrix $\tilde{L}$ couples the updates of all $y_i$ and thus the resolvent of $\optT$ can not be computed distributedly, we introduce a design matrix $\Phi \coloneqq \btau^{-1} - \rho\tilde{L}$ and compute the resolvent of $\Phi^{-1}\optT$ instead, 
where $\btau \coloneqq \blkd{\btau_1, \ldots, \btau_N}$ and each  $\btau_i \coloneqq \blkd{\tau_{i0}\otimes I_{n_i}, \{\tau_{ij}\otimes I_{n_j}\}_{j \in \neighbN{+}{i}}} \in \pdset{n_i+n^+_i}{++}$ is a diagonal matrix with all diagonal entries (step sizes) positive. 
Moreover, to ensure that the zero set of $\optT$ and $\Phi^{-1}\optT$ are equivalent and the associated Hilbert space is well-defined, the selected step sizes in $\btau$ should be sufficiently small such that $\Phi$ is positive definite. 
According to the Gershgorin circle theorem (\cite{bell1965gershgorin}), for each player $i$, it suffices to set the step size of the local decision $\tau^{-1}_{i0} > 2\rho N^-_i$ and the step sizes of the local estimate $\tau^{-1}_{ij} > 2$. 
Let $\pspace$ be the Hilbert space obtained by endowing the vector space $\rset{\tilde{n}}{}$ with the inner product $\langle y, y' \rangle_\pspace = \langle \Phi y, y'\rangle$. 
With all the introduced elements, we can compute a zero of $\Phi^{-1}\optT$ by utilizing the following iteration: 
\begin{align}\label{eq:ppa-km}
\tilde{y}^{(k+1)} \coloneqq J_{\Phi^{-1}\optT}(y^{(k)}), \;y^{(k+1)} \coloneqq y^{(k)} + \gamma^{(k)}(\tilde{y}^{(k+1)} - y^{(k)}). 
\end{align}

The detailed implementation of \eqref{eq:ppa-km} consists of the optimization of augmented best-response objectives $\hat{\mathbb{J}}^{(k)}_i$'s and some linear updates, where $\hat{\mathbb{J}}^{(k)}_i(\tilde{y}^i_i; w^*_i) \coloneqq \mathbb{J}_i(\tilde{y}^i_i; \tilde{y}^{+(k+1)}_i, w^*_i) + \rho(\sum_{{j \in \neighbN{-}{i}}}{y}^{i(k)}_{i} - {y}^{i(k)}_{j})^T\tilde{y}^i_i + \frac{1}{2\tau_{i0}}\norm{\tilde{y}^{i}_{i} - {y}^{i(k)}_{i}}^2_2$,
which is omitted here for brevity.
In the following, we describe a modified version that will be used throughout the learning dynamics in Section~\ref{subsec:sne-mis-param}. 
With the introduction of the local estimates $y^j_i$ and the extended pseudogradient $\extgjacob_{w^*}$, the operator $\optT$ is no longer maximally monotone and the resolvent $J_{\Phi^{-1}\optT}$ does not possess the firmly nonexpansive property in general. 
We denote the greatest (resp. smallest) out-neighbor count in $\mathcal{G}$ by $\bar{N}^-$ (resp. $\ubar{N}^-$), i.e., $\bar{N}^- \coloneqq \max\{N^-_i: i \in \playerN\}$ (resp. $\ubar{N}^-\coloneqq \min\{N^-_i: i \in \playerN\}$). 
To prove the convergence for the iteration with precise parameters, we will need to impose the following assumption on the pseudo-gradient $\gjacob_{w^*}$ and the extended pseudo-gradient $\extgjacob_{w^*}$. 
\begin{assumption}{(Regularity of Pseudo-Gradient)}\label{asp:convg}
At least one of the following statements holds:

(i) The operator $\mathcal{R}^T\extgjacob_{w^*}+ \rho \tilde{L}$ is maximally monotone;

(ii) The pseudogradient $\gjacob_{w^*}$ is strongly monotone and Lipschitz continuous, i.e., there exist $\eta > 0$ and $\theta_1 > 0$, such that $\forall x, x' \in \rset{n}{}$, $\langle x - x', \gjacob_{w^*}(x) - \gjacob_{w^*}(x')\rangle \geq \eta \norm{x - x'}^2$ and $\norm{\gjacob_{w^*}(x) - \gjacob_{w^*}(x')} \leq \theta_1\norm{x - x'}$. The operator $\mathcal{R}^T\extgjacob_{w^*}$ is Lipschitz continuous, i.e., there exists $\theta_2 > 0$, such that $\forall y, y' \in \rset{nN}{}$, $\norm{\extgjacob_{w^*}(y) - \extgjacob_{w^*}(y')} \leq \theta_2 \norm{y - y'}$. 
Moreover, the weight of $\rho\tilde{L}$ satisfies $\rho \geq \frac{1}{\sigma_1}\big(\frac{\bar{N}^- + 1}{\ubar{N}^- + 1}\frac{(\theta_1 + \theta_2)^2}{4\eta} + \theta_2\big)$. 
\end{assumption}

From one perspective, the convergence result under Assumption~\ref{asp:convg} (i) directly follows from the monotone operator theory and firmly nonexpansive fixed-point iterations. 
Nevertheless, verifying the fulfillment of this assumption is cumbersome and often can only be done numerically (See \cite[Sec.~4.2.3]{johansson2012distributed} for examples).
From another perspective, for a monotone game, i.e., games with  monotone $\gjacob_{w^*}$, when restricted to the consensus subspace $y^j_i = y^j_j$ for all $i \in \playerN$ and $j \in \neighbN{+}{i}$, $\optT$ enjoys maximally monotonicity and its resolvent $J_{\Phi^{-1}\optT}$ possesses firmly nonexpansiveness. 
Controlling the growth rate of $\extgjacob_{w^*}$ w.r.t. local estimates with Lipschitz continuity, Assumption~\ref{asp:convg} (ii) extends the above statement and allows for a violation of the consensus constraints if the missing monotonicity of $\extgjacob$ can be compensated by the measure of violation $\frac{\rho}{2}y^T\tilde{L}y$. 
A weaker concept that emerges under (ii) is quasinonexpansiveness: a general operator $A$ is called quasinonexpansive if $\forall x \in \dom{A} $ and $\forall y \in \fix{A}$, $\norm{Ax - y} \leq \norm{x - y}$. 
Now we are ready to formulate the theorem providing sufficient conditions for the proposed algorithm to work with perfect information of its model parameters. 
The proof is reported in \cite[Appendix~A]{huang2021distlearn}. 

\begin{theorem}\label{thm:exact-convg-pf}
Suppose Assumptions~\ref{asp:commtopo} to \ref{asp:convg} hold, and $\btau$ is properly chosen such that $\Phi$ is positive definite.
Then $J_{\Phi^{-1}\optT}$ is a continuous quasinonexpansive operator, and the sequence $(y^{(k)})_{k \in \nset{}{}}$ generated by the fixed-point iteration \eqref{eq:ppa-km} will converge to a zero of $\optT$, where $\gamma^{(k)} \in [0, 1]$ and $\sum_{n \in \nset{}{}} \gamma^{(k)}(1 - \gamma^{(k)}) = +\infty$.
\end{theorem}

\subsection{SNE Problems with Unknown Parameters}\label{subsec:sne-mis-param}

\begin{algorithm2e}
\SetAlgorithmName{Subroutine}{subroutine}{List of Subroutines}
\SetAlgoLined
\caption{Distributed Nash Equilibrium Seeking}\label{alg:node-edge}
\textbf{Available variables:} $\{{y}^{(k)}_{i}\}$, $\{\hat{w}^{(k)}_i\}$ \;
\textbf{At the $k$-th iteration, each player $i \in \playerN$:}\\
\qquad Receive $\{{y}^{j(k)}_j\}_{j \in \neighbN{+}{i}}$ from its in-neighbors and $\{{y}^{i(k)}_j\}_{j \in \neighbN{-}{i}}$ from its out-neighbors\;
\qquad $\tilde{y}^{j(k+1)}_{i} = {y}^{j(k)}_{i} - \tau_{ij}\rho({y}^{j(k)}_{i} - {y}^{j(k)}_{j}), \quad \forall j \in \neighbN{+}{i}$\;
\qquad $\tilde{y}^{i(k+1)}_{i} = \underset{\tilde{y}^{i}_{i} \in \mathcal{X}_i}{\argmin}\Big\{ \mathbb{J}_i(\tilde{y}^i_i; \tilde{y}^{+(k+1)}_i, \hat{w}^{(k)}_{i}) + \rho(\underset{j \in \neighbN{-}{i}}{\sum}{y}^{i(k)}_{i} - {y}^{i(k)}_{j})^T\tilde{y}^i_i + \frac{1}{2\tau_{i0}}\norm{\tilde{y}^{i}_{i} - {y}^{i(k)}_{i}}^2_2 \Big\}$\;
\qquad ${y}_i^{(k+1)} = {y}^{(k)}_i + \gamma^{(k)}(\tilde{y}^{(k+1)}_i - {y}^{(k)}_i)$.
\end{algorithm2e}

We now shift to the setting where at each iteration $k$, each player $i$ has no access to the precise parameter $w^*_i$ while it maintains a parameter estimate $\hat{w}^{(k)}_i \coloneqq [\hat{w}^{(k)}_{ji}]_{j \in \{i\} \cup \neighbN{+}{i}}$. 
We investigate the following simple learning dynamics where at each iteration, each player $i$ first makes a decision which is determined by the optimizer of the given augmented objective using the estimated parameters and certain random exploration factor, then observes the realized objective value and the decisions made by its neighbors, and finally updates the estimates the unknown parameters. 
Note that instead of minimizing the augmented objective $\hat{\mathbb{J}}^{(k)}_i(\tilde{y}^i_i; w^*_i)$, player $i$ now makes a best-response decision w.r.t. $\hat{\mathbb{J}}^{(k)}_i(\tilde{y}^i_i; \hat{w}^{(k)}_i)$. 
This substitution in parameters gives rise to an estimated operator $\optT^{(k)}$ for $\optT$, with $\extgjacob_{w^*}$ replaced by $\extgjacob_{\hat{w}^{(k)}}$. 
For notational brevity, we let the exact iterations be denoted by $\mathscr{R}_* \coloneqq J_{\Phi^{-1}\optT}$ and $\mathscr{P}_* \coloneqq I + \gamma^{(k)}(\mathscr{R}_* - I)$, and the estimated iterations (based on parameter estimates) be denoted by  $\mathscr{R}^{(k)} \coloneqq J_{\Phi^{-1}\optT^{(k)}}$ and $\mathscr{P}^{(k)} \coloneqq I + \gamma^{(k)}(\mathscr{R}^{(k)} - I)$. 
The SNE seeking dynamics is then described by $y^{(k+1)} \coloneqq \mathscr{P}^{(k)}y^{(k)}$, and the intermediate result is given by $\tilde{y}^{(k+1)} \coloneqq \mathscr{R}^{(k)}y^{(k)}$. 
The detailed implementation is included in Subroutine~\ref{alg:node-edge}.

\section{Parameter Learning Model}

In this section, we consider the proper way to generate the estimate sequence $(\hat{w}^{(k)}_i)_{k \in \nset{}{}}$ when each player $i$ has no clue or is uncertain about the parameters $w^{*}_{i}$ inside its objective $\mathbb{J}_i$. 
Assume that each player has access to a bandit feedback system, which returns the realized objective function value based on the decision profile of the whole player network. 
To enable each player to perform ordinary least squares estimation to learn $\hat{w}^{(k)}_i$ at each iteration $k$, we make the following assumption:
\begin{assumption}\label{asp:olse} (Parameter Learning)
For each player $i \in \playerN$ and at each iteration $k \in \nset{}{}$, the following conditions hold:

(i) $(\xi^{(k)}_{i})_{k \in \nset{}{}}$ is a sequence of real independent random variables with expectation zero and range bounded;


(ii) The scenario-based function $J_i$ is invertible in $s_i$; 

(iii) The feasible parameter set $\mathcal{W}_i \subseteq \rset{n^+_i+1}{}$ is convex and compact. In addition, for any $\hat{w}_i \in \mathcal{W}_i$, the augmented objective $\hat{\mathbb{J}}^{(k)}_i(\cdot; \hat{w}_i)$ is a strictly convex function on $\mathcal{X}_i$. 
\end{assumption}

Note that the uniformly ``thin'' tail of the random variable $\xi_{i}$ is essential to apply the large deviations result for later convergence analysis, and Assumption~\ref{asp:olse} (i) prescribes a sufficient condition where this requirement is satisfied. 
Condition (ii) enables each player to recover the values $s_i$ based on the observed objective values. 
Condition (iii) ensures that $\hat{\mathbb{J}}^{(k)}_i(\cdot; \hat{w}_i)$ with $\hat{w}_i \in \mathcal{W}_i$ admits a unique $\argmin$ solution and the resolvent $J_{\Phi^{-1}\optT}$ is well-defined and single-valued in $\pspace$. 

To estimate the unknown parameters, each player $i \in \playerN$ picks the pivot point ${y}^{i(k+1)}_{i} \in \mathcal{X}_i$ where this player seeks to observe its payoff and estimate the parameters. 
Moreover, player $i$ draws a random exploration vector $\delta^{(k)}_{i}: \Omega \to \rset{n_i}{}$ and plays $\check{y}^{i(k+1)}_{i} \coloneqq {y}^{i(k+1)}_i + \delta^{(k)}_{i}$.
The random vector $\delta^{(k)}_{i}$ should satisfy the following assumption. 

\begin{assumption}\label{asp:rdm-explr}(Random Exploration)
For each player $i\in \playerN$, $(\delta^{(k)}_{i})_{k \in \nset{}{}}$ is a sequence of independent identically distributed (i.i.d.) random variables with zero mean, bounded range, and positive definite covariance matrix. 
\end{assumption}

Nevertheless, the feasibility issue will arise with the introduction of the random exploration $\delta^{(k)}_i$. 
In the spirit of \cite{agarwal2010optimal, bravo2018bandit}, we assume in the following that each local feasible set $\mathcal{X}_i$ is a convex body in $\rset{n_i}{}$, i.e., it has a nonempty topological interior. 
Moreover, we will introduce a "safe net", and adjust the chosen pivot point ${y}^{i(k+1)}_i$ to reside within a suitably shrunk zone of $\mathcal{X}_i$. 
In details, let $\mathbb{B}_{r_i}(p_i)$ be an $r_i$-ball centered at some $p_i \in \mathcal{X}_i$ so that $\mathbb{B}_{r_i}(p_i) \subseteq \mathcal{X}_i$. 
Then, instead of directly perturbing ${y}^{i(k+1)}_i$ by $\delta^{(k)}_{i}$, we consider the feasibility adjustment
$\tilde{\delta}^{(k)}_{i} \coloneqq \delta^{(k)}_{i} - \bar{\delta}_ir_i^{-1}({y}^{i(k+1)}_i - p_i)$, 
where $\bar{\delta}_i \coloneqq \max\{\norm{\delta_i(\omega)}_2: \omega \in \Omega\}$ satisfies $\bar{\delta}_ir_i^{-1} < 1$. 
Each player $i$ plays ${y}^{i(k+1)}_{i} + \tilde{\delta}^{(k)}_{i}$ instead of ${y}^{i(k+1)}_i + \delta^{(k)}_{i}$. 
This adjustment moves each pivot $\mathcal{O}(\bar{\delta}_i)$-closer to the interior base point $p_i$ with  ${y}^{i(k+1)}_{i\delta} = {y}^{i(k+1)}_i - \bar{\delta}_ir_i^{-1}({y}^{i(k+1)}_i - p_i)$, and then perturbs ${y}^{i(k+1)}_{i\delta}$ by $\delta^{(k)}_{i}$. 
Given the fact that $\check{y}^{i(k+1)}_{i} \coloneqq {y}^{i(k+1)}_{i\delta} + \delta^{(k)}_{i} = (1 - \bar{\delta}_ir_i^{-1}){y}^{i(k+1)}_i + \bar{\delta}_ir_i^{-1}(p_i + r_i\bar{\delta}_i^{-1}\delta^{(k)}_{i}) \in \mathcal{X}_i$ and $p_i + r_i\bar{\delta}_i^{-1}\delta^{(k)}_{i} \in \mathbb{B}_{r_i}(p_i)$, feasibility of the query point is then ensured. 
After the above feasibility adjustment, let $s^{(k)}_{i} = w^*_{ii} + \sum_{j \in \neighbN{+}{i}}w^{*T}_{ji}\check{y}^{j(k+1)}_{j} + \xi^{(k)}_{i}$ and $\ell^{(k)}_i = [1; [\check{y}^{j(k+1)}_{j}]_{j \in \neighbN{+}{i}}]$. 
Based on the observed values available at the $k$-th iteration, the OLSE $\hat{w}^{(k+1)}_i$ is given by 
\begin{align}\label{eq:olse-linear}
\hat{w}^{(k+1)}_i \coloneqq \argmin_{w_i \in \mathcal{W}_i} \frac{1}{k+1}\sum_{t=0}^{k}
(s^{(t)}_{i} - \langle \ell^{(t)}_{i}, w_i\rangle)^2.
\end{align}
The complete parameter learning dynamics is given in Subroutine~\ref{alg:param-est}. 


\begin{algorithm2e}
\SetAlgorithmName{Subroutine}{subroutine}{List of Subroutines}
\SetAlgoLined
\textbf{At the $k$-th iteration:}\\
Each player $i \in \playerN$: \\
\qquad Randomly picks an exploration factor $\delta^{(k)}_{i}$;\\
\qquad Makes its decision to play $\check{y}^{i(k+1)}_{i} \coloneqq {y}^{i(k+1)}_i + \delta^{(k)}_{i} - \bar{\delta}_ir_i^{-1}({y}^{i(k+1)}_i - p_i)$; \\
\qquad Observes $J_i(\check{y}^{i(k+1)}_{i}; s^{(k)}_{i})$ and receives $\{\check{y}^{j(k+1)}_{j}\}_{j \in \neighbN{+}{i}}$ from its in-neighbors; \\
\qquad Estimates the unknown parameters $\hat{w}^{(k+1)}_i$ by solving \eqref{eq:olse-linear}. 
\caption{Unknown Parameter Estimation}
\label{alg:param-est}
\end{algorithm2e}


\section{Learning Dynamics and Convergence Analysis}\label{sec:convg-analy}

Assembling the updating steps of NE seeking and those of parameter estimation together, the learning dynamics of NE seeking with unknown parameters in objectives is described in Algorithm~\ref{alg:assemble}.

\begin{algorithm2e}
\SetAlgoLined
\textbf{Initialize:} $\{{y}^{(0)}_i\}$, $\{\hat{w}^{(0)}_i\}$ with ${y}^{(0)}_i \in \mathcal{X}_i$ and $\hat{w}^{(0)}_i \in \mathcal{W}_i$\; 
\textbf{Iterate until convergence}: \\
1) NE seeking updating step: run Subroutine~\ref{alg:node-edge}\;
2) parameter estimation updating step: run Subroutine~\ref{alg:param-est}\;
\textbf{Return:} $\{{y}^{(k)}_i\}$, $\{\hat{w}^{(k)}_i\}$. 
\caption{Distributed Learning of v-SGNE with Unknown Parameters}
\label{alg:assemble}
\end{algorithm2e}

We start by establishing the following convergence result for a fixed-point iteration with the K-M scheme and a general continuous quasinonexpansive fixed-point iteration operator $\mathscr{R}_*:\mathcal{H} \to \mathcal{H}$ and its approximates $(\mathscr{R}^{(k)})_{k \in \nset{}{}}$:
\begin{align}\label{eq:fxpt-iter-approx}
    x^{(k+1)} \coloneqq x^{(k)} + \gamma^{(k)}(\mathscr{R}^{(k)}x^{(k)} - x^{(k)}), 
\end{align}
where $\mathcal{H}$ is a finite-dimensional Hilbert space, with its inner product and norm denoted by $\langle, \rangle_{\mathcal{H}}$ and $\norm{\cdot}_{\mathcal{H}}$, respectively. 
Before proceeding, we introduce the following notations to facilitate the later discussion and analysis.
The estimated iteration error and its norm in $\mathcal{H}$ at each iteration $k$ are defined as: 
$\epsilon^{(k)} \coloneqq \mathscr{R}^{(k)}(x^{(k)}) - \mathscr{R}_*(x^{(k)})$ and $\varepsilon^{(k)} \coloneqq \norm{\epsilon^{(k)}}_\mathcal{H}$. 
A residual function $\res{x} \coloneqq \norm{x - \mathscr{R}_*(x)}_\mathcal{H}$ is introduced such that $\res{x^*} = 0$ is a necessary condition for $x^* \in \fix{\mathscr{R}_*}$. 
The proof of the following convergence theorem is reported in \cite[Appendix~B]{huang2021distlearn}. 

\begin{theorem}\label{thm:main-convg-thm}
Let $(\Omega, \mathcal{F}, \mathcal{P})$ be a probability space and $\mathcal{F}_0 \subseteq \mathcal{F}_1 \subseteq \mathcal{F}_2 \subseteq \cdots$ be a sequence of sub-$\sigma$-fields of $\mathcal{F}$.  
Suppose for $k=0, 1, \ldots$, $x^{(k)}$ is an $\mathcal{F}_k$-measurable random vector generated by the inexact fixed-point iteration \eqref{eq:fxpt-iter-approx}, where $\mathscr{R}_*:\mathcal{H} \to \mathcal{H}$ is a continuous quasinonexpansive operator and $(\mathscr{R}^{(k)})_{k \in \nset{}{}}$ denotes a sequence of its approximates subject to stochasticity. 
Moreover, suppose the sequence $(\gamma^{(k)})_{k \in \nset{}{}}$ satisfies $0 \leq \gamma^{(k)} \leq 1$ and $\sum_{k \in \nset{}{}}\gamma^{(k)}(1 - \gamma^{(k)}) = +\infty$. 
If the approximates $(\mathscr{R}^{(k)})_{k \in \nset{}{}}$ and the generated $(x^{(k)})_{k \in \nset{}{}}$ satisfy the following two conditions:

(i) $(\norm{x^{(k)}}_\mathcal{H})_{k \in \nset{}{}}$ is bounded a.s.;
(ii) $\sum_{k \in \nset{}{}} \gamma^{(k)}\expt{}{\varepsilon^{(k)} \mid \mathcal{F}_{k}} < +\infty$ a.s., \\
then $(x^{(k)})_{k \in \nset{}{}}$ will almost surely converge to a fixed point of $\mathscr{R}_*$. 
\end{theorem}

We now consider the specific iteration for the locally coupled network games discussed in Section~\ref{subsec:sne-mis-param}. 
Let $(y^{(k)})_{k \in \nset{}{}}$ denote the sequence generated by applying $(\mathscr{P}^{(k)})_{k \in \nset{}{}}$, i.e., $y^{(k + 1)} \coloneqq \mathscr{P}^{(k)}\circ \cdots \circ \mathscr{P}^{(0)}(y^{(0)})$. 
We further define $y^{(k+1)}_* \coloneqq \mathscr{P}_*(y^{(k)})$ for each $k \in \nset{}{}$. 
For brevity, we shall write $\{{\delta}^{(k)}_i\}$ in replacement of $\{{\delta}^{(k)}_i\}_{i \in \playerN}$ and similarly for other sets indexed by $\playerN$, unless otherwise specified. 
For each $k \geq 2$, define the sub-$\sigma$-field $\mathcal{F}_{k}$ as follows:
\begin{align}
\mathcal{F}_{k} \coloneqq \sigma\{y^{(0)}, \hat{w}^{(0)}, \{{\delta}^{(0)}_i\}, \{{\xi}^{(0)}_i\}, \ldots, \{{\delta}^{(k-2)}_i\}, \{{\xi}^{(k-2)}_i\}\};
\end{align}
and define $\mathcal{F}_{0}\coloneqq \sigma\{y^{(0)}\}$ and $\mathcal{F}_{1}\coloneqq \sigma\{y^{(0)}, \hat{w}^{(0)}\}$. 
In the next theorem, we will prove that the sequence produced by Algorithm~\ref{alg:assemble} satisfies the conditions required in Theorem~\ref{thm:main-convg-thm}. 
Define $\norm{\Delta \hat{w}^{(k)}_i}_2 \coloneqq \norm{\hat{w}^{(k)}_i - w^*_i}_2$, $\norm{\Delta \hat{w}^{(k)}}_2 \coloneqq \norm{\hat{w}^{(k)} - w^*}_2$, and the estimated iteration error $\epsilon^{(k)} \coloneqq \mathscr{R}^{(k)}(y^{(k)}) - \mathscr{R}_*(y^{(k)})$ and its norm $\varepsilon^{(k)} \coloneqq \norm{\epsilon^{(k)}}_\pspace$, with $\mathscr{R}^{(k)}$ and $\mathscr{R}_*$ given in Sec.~\ref{subsec:sne-mis-param}. 
The detailed proof is reported in \cite[Appendix~B]{huang2021distlearn}. 

\begin{theorem}\label{thm:bdd-est-seq-sumb}
Consider the sequence $(y^{(k)})_{k \in \nset{}{}}$ generated by Algorithm~\ref{alg:assemble}. 
Suppose Assumptions~\ref{asp:commtopo} to \ref{asp:olse} hold, and given a sequence $(\gamma^{(k)})$ where $0 \leq \gamma^{(k)} \leq 1$ and $\sum_{k \in \nset{}{}}\gamma^{(k)}(1 - \gamma^{(k)}) = +\infty$, the sequence $(\gamma^{(k)}\expt{}{\norm{\Delta\hat{w}^{(k)}}_2 \mid \mathcal{F}_k})_{k \in \nset{}{}}$ is a.s. absolutely summable. 
Then $(\norm{y^{(k)}}_{\pspace})_{k \in \nset{}{}}$ is bounded a.s., and the estimated iteration error satisfies $\sum_{k \in \nset{}{}} \gamma^{(k)}\expt{}{\varepsilon^{(k)} \mid \mathcal{F}_{k}} < \infty$ a.s.
\end{theorem}

To establish the convergence of Algorithm~\ref{alg:assemble}, Theorems~\ref{thm:main-convg-thm} and \ref{thm:bdd-est-seq-sumb} suggest that the sequence $(\supsub{\gamma}{(k)}{}\expt{}{\norm{\Delta\supsub{\hat{w}}{(k)}{}}_2 \mid \mathcal{F}_k})_{k \in \nset{}{}}$ should fulfill the summability assumption a.s. 
Our aim in what follows will be proving the following asymptotic convergence rate result of OLSE and investigating the relation between the estimation error $\norm{\Delta\supsub{\hat{w}}{(k)}{}}_2$ and the total number of observations made until the $k$-th iteration by utilizing the law of large deviation for OLSEs. 
We refer the interested readers to \cite[Appendix~C]{huang2021distlearn} for the detailed proof. 

\begin{theorem}\label{thm:obsv-num-est-err}
Suppose Assumptions \ref{asp:olse} and \ref{asp:rdm-explr} hold, 
and each player $i \in \playerN$ at the iteration $k \in \nset{}{}$ uses Subroutine~\ref{alg:param-est} to obtain an estimate $\hat{w}^{(k)}_i$. 
Let $\alpha_2 \in (0, \frac{1}{2})$ be an arbitrary constant. 
Then on a sample set $\hat{\Omega}$ with probability one, for any $\hat{\omega} \in \hat{\Omega}$, there exists a sufficiently large index $K_{\text{lse}}(\hat{\omega}) \in \nset{}{}$ such that for all $k > K_{\text{lse}}(\hat{\omega})$, $\expt{}{\norm{\Delta \hat{w}^{(k)}_{i}}_2 \mid \mathcal{F}_k}(\hat{\omega}) \leq C_{\Delta,i}k^{\alpha_2 - 1/2}$ holds for each $i \in \playerN$, where $C_{\Delta,i}$ is a constant independent of $\hat{\omega}$ and $k$.
\end{theorem}

To conclude, we note that as long as the sequence of step sizes is chosen as $\gamma^{(k)} \coloneqq 1/k^{\alpha_1}$ with $1/2 < \alpha_1 \leq 1$, there always exists a feasible $\alpha_2 = \frac{1}{2}(\alpha_1 - \frac{1}{2}) \in (0, \frac{1}{4}]$, such that by Theorem~\ref{thm:obsv-num-est-err} for any $\hat{\omega} \in \hat{\Omega}$ and $k > K_{\text{lse}}(\hat{\omega})$, $\gamma^{(k)}\expt{}{\norm{\Delta \hat{w}^{(k)}_{i}}_2 \mid \mathcal{F}_k}(\hat{\omega}) \leq C_{\Delta,i}k^{-\frac{1}{2}\alpha_1 - \frac{3}{4}}$ for all $i$.
This together with Theorems~\ref{thm:main-convg-thm} and \ref{thm:bdd-est-seq-sumb} implies the almost-sure convergence of Algorithm~\ref{alg:assemble}, i.e., $y^{i(k)}_i \overset{\text{a.s.}}{\to} x^*_i$ and $\hat{w}^{(k)}\overset{\text{a.s.}}{\to} w^*_i$ for all $i$, where $x^* \coloneqq [x^*_i]_{i \in \playerN}$ denotes an SNE of \eqref{eq:pf-optprob}.

\section{Conclusion and Future Directions}

This paper develops a distributed solution to find Nash equilibria in stochastic locally coupled network games with unknown parameters by combining the proximal-point algorithm for Nash equilibrium seeking and the ordinary least square estimator for parameter learning. 
Almost-sure convergence of the solution algorithm is established, which can be further extended to handle generalized Nash equilibrium problems and iterations using inexact solvers. 
There remain several open problems. 
In the learning dynamics, to fulfill the identifiability condition for the estimator, each player is required to add random exploration factors to its decisions, and the actual decisions (perturbed by random exploration factors) it plays throughout the iteration will eventually bounce within some $\epsilon$-neighborhood of a true Nash equilibrium, instead of converging to it. 
Hence, one of our future directions is to design learning dynamics such that the actual sequences of play can converge to the true Nash equilibria. 
Another potential future direction resides in considering an estimator which can better deal with the nonlinear parameter estimation and can work more efficiently in an online-learning fashion with suitable guarantees on convergence rate. 
In addition, even though we can extend the current analysis and similarly prove the convergence to a generalized Nash equilibrium when taking locally coupled constraints and global resource constraints, the actual action sequence may violate these coupled constraints during the iterations, which prevents the application of the proposed solution in some practical situations. 
We intend to address these questions in future work. 

\acks{This work was supported by the National Science Foundation under Grant No. 2014816 and No. 2038410. }

\bibliography{references}

\pagebreak
\section*{Appendix}
\renewcommand{\thesubsection}{\Alph{subsection}}

\subsection{Reformulation of the Problem and Convergence Analysis under Accurate Knowledge of Parameters}\label{appdx:eqv-prob}

\begin{proof}[Proof of Theorem~\ref{thm:eq-prob}]

Suppose there exists $y^* \in \zer{\optT}$. 
We have $\mathcal{R}^T\extgjacob_{w^*}(y^*) + N_{\tilde{\mathcal{X}}}(y^*) + \rho\tilde{L}y^* \ni \bzero$. 
For the rows corresponding to the local estimates $y^j_i$, since $\mathcal{R}^T\extgjacob_{w^*}(y^*)$ and $N_{\tilde{\mathcal{X}}}(y^*)$ do not involve local estimates, we have $y^{j*}_i = y^{j*}_j$. 
Under these consensus results, it can be shown
\begin{align*}
\partial_{x_i}\mathbb{J}_i(y^{i*}_i; \{y^{j*}_j\}_{j \in \neighbN{+}{i}}, w^*_i) + N_{\mathcal{X}_i}(y^{i*}_i) \ni \bzero, \forall i \in \playerN,
\end{align*}
and our claim follows. 
In the other direction, if $\{y^{i\dagger}_i\}_{i \in \playerN}$ is an NE of problem \eqref{eq:pf-optprob}, we can set each local estimate $y^{j\dagger}_i = y^{j\dagger}_j$ for all $i \in \playerN$ and $j \in \neighbN{+}{i}$, and the resulting stack $y^\dagger \in \zer{\optT}$. 
\end{proof}

\begin{proof}[Proof of Theorem~\ref{thm:exact-convg-pf}]

Under Assumption~\ref{asp:convg} (i), by the monotone operator theory, it is straightforward to verify that the resolvent $J_{\Phi^{-1}\optT}$ is firmly nonexpansive, and the proposed fixed-point iteration converges to a fixed point of $\optT$. 
Here, we focus on proving the theorem under Assumption~\ref{asp:convg} (ii). 

We begin by proving the weaker sense of monotonicity that $\optT$ possesses. 
For arbitrary $y \in \mathcal{X}$ and $y^* \in \zer{\optT}$ and their associated graphs w.r.t. $\mathcal{R}^T\extgjacob_{w^*}$, i.e., $(y, g), (y^*, g^*) \in \gra{\mathcal{R}^T\extgjacob_{w^*}}$, we have
\begin{align*}
\begin{split}
\langle y - y^*, g + \rho\tilde{L}y - g^* - \rho\tilde{L}y^*\rangle = \langle y - y^*, g - g^*\rangle + \langle y - y^*, \rho\tilde{L}(y - y^*) \rangle. 
\end{split}
\end{align*}
We construct a row switching matrix $\mathcal{S}_r$, which, for each player $i$, groups $y^i_i$ and $y^i_j$ for all $j \in \neighbN{-}{i}$ together, i.e., $\mathcal{S}_ry = [y^i_i; [y^i_j]_{j \in \neighbN{-}{i}}]_{i \in \playerN}$. 
Moreover, the block-wise consensus matrix $\mathcal{C}_B$ after the switching can be defined as:
\begin{align*}
\mathcal{C}_B \coloneqq \blkd{\{\frac{1}{1+N^-_i}\bone_{(1+N^-_i) \times (1+N^-_i)}\otimes I_{n_i}\}_{i \in \playerN}}. 
\end{align*}
Decompose $y$ into two orthogonal components, i.e., $y = y^\perp + y^\parallel$, where $y^\parallel = \mathcal{S}_r^T\mathcal{C}_B\mathcal{S}_ry$ and $y^\perp = (I - \mathcal{S}_r^T\mathcal{C}_B\mathcal{S}_r)y$. 
Since $y^* \in \zer{\optT}$, $y^* = y^{*\parallel}$. 
Let $\sigma_1$ denote the smallest positive eigenvalue of $\tilde{L}$. 
As a result, $\langle y - y^*, \rho\tilde{L}(y - y^*) \rangle \geq \sigma_1 \rho \norm{y^\perp}^2$, and
\begin{align*}
\begin{split}
\langle y - y^*, g - g^*\rangle &= \langle y^\perp + y^\parallel - y^*, g - g^\parallel + g^\parallel - g^*\rangle \\ 
&= \langle y^\perp, g - g^\parallel \rangle + \langle y^\perp, g^\parallel - g^* \rangle + \langle y^\parallel - y^*, g - g^\parallel\rangle + \langle y^\parallel - y^*, g^\parallel - g^*\rangle \\
&\geq \frac{\eta}{\bar{N}^- + 1}\norm{y^\parallel - y^*}^2 - \frac{\theta_1 + \theta_2}{\sqrt{\ubar{N}^- + 1}}\norm{y^\perp}\norm{y^\parallel - y^*} - \theta_2\norm{y^\perp}^2,
\end{split}
\end{align*}
where, in the first line, we let $(y^\parallel, g^\parallel) \in \gra{\mathcal{R}^T\extgjacob}$; in the last line, we use $\bar{N}^-$ to denote the greatest out-neighbor count in this network, i.e., $\bar{N}^- \coloneqq \max\{N^-_i: i \in \playerN\}$, and $\ubar{N}^-$ to denote the smallest out-neighbor count, i.e., $\ubar{N}^- \coloneqq \min\{N^-_i: i \in \playerN\}$. 
Merging the above two inequalities, we obtain: 
\begin{align}
\langle y - y^*, g - g^* + \rho\tilde{L}(y - y^*)\rangle \geq 
\begin{bmatrix} \norm{y^\parallel - y^*} \\ \norm{y^\perp} \end{bmatrix}^T
\begin{bmatrix}
\frac{\eta}{\bar{N}^- + 1} & -\frac{\theta_1 + \theta_2}{2\sqrt{\ubar{N}^- + 1}} \\
-\frac{\theta_1 + \theta_2}{2\sqrt{\ubar{N}^- + 1}} & \sigma_1\rho - \theta_2
\end{bmatrix}
\begin{bmatrix} \norm{y^\parallel - y^*} \\ \norm{y^\perp} \end{bmatrix} \geq 0.
\end{align}
when $\rho \geq \frac{1}{\sigma_1}\big(\frac{\bar{N}^- + 1}{\ubar{N}^- + 1}\cdot\frac{(\theta_1 + \theta_2)^2}{4\eta} + \theta_2\big)$.
It further implies that when $\rho$ satisfies this inequality, the operator $\Phi^{-1}\optT$ enjoys the similar restricted monotonicity w.r.t. its fixed points, i.e., $\forall y \in \mathcal{X}$ and $\forall y^* \in \zer{\optT} = \zer{\Phi^{-1}\optT}$, along with $(y, \tilde{g}), (y^*, \tilde{g}^*) \in \gra{\Phi^{-1}\optT}$, $\langle y - y^*, \tilde{g} - \tilde{g}^* \rangle_\pspace = \langle  y - y^*, \Phi g - \Phi g^* \rangle \geq 0$, by the monotonicity of the normal cones of local feasible sets. 
Moreover, we also have $y - J_{\Phi^{-1}\optT}(y) \in \Phi^{-1}\optT(J_{\Phi^{-1}\optT}(y))$ and $y^* - J_{\Phi^{-1}\optT}(y^*) \in \Phi^{-1}\optT(J_{\Phi^{-1}\optT}(y^*))$. 
By the restricted monotonicity of $\optT$ shown above, we derive
\begin{align*}
\langle (y - J_{\Phi^{-1}\optT}(y)) - (y^* - J_{\Phi^{-1}\optT}(y^*)), J_{\Phi^{-1}\optT}(y) - J_{\Phi^{-1}\optT}(y^*)\rangle_\pspace \geq 0, 
\end{align*}
which proves that $J_{\Phi^{-1}\optT}$ is quasinonexpansive \citeApndx[Prop.~4.4]{BauschkeHeinzH2017CAaM}.

The iteration \eqref{eq:ppa-km} suggests the following recursive relationship: 
\begin{align*}
& \norm{y^{(k+1)} - y^*}^2_\pspace = \norm{\gamma^{(k)}(J_{\Phi^{-1}\optT}(y^{(k)}) - y^*) + (1-\gamma^{(k)})(y^{(k)} - y^*)}^2_\pspace \\
&= \gamma^{(k)}\norm{J_{\Phi^{-1}\optT}(y^{(k)}) - y^*}^2_\pspace + (1 - \gamma^{(k)})\norm{y^{(k)} - y^*}^2_\pspace - \gamma^{(k)}(1 - \gamma^{(k)})\norm{J_{\Phi^{-1}\optT}(y^{(k)}) - y^{(k)}}^2_\pspace \\
& \leq \norm{y^{(k)} - y^*}^2_\pspace - \gamma^{(k)}(1 - \gamma^{(k)})\norm{J_{\Phi^{-1}\optT}(y^{(k)}) - y^{(k)}}^2_\pspace,
\end{align*}
which implies that $(y^{(k)})_{k \in \nset{}{}}$ is Fejer monotone w.r.t. $\fix{J_{\Phi^{-1}\optT}}$. 
The above relation further suggests that $\sum_{k \in \nset{}{}}\gamma^{(k)}(1 - \gamma^{(k)})\norm{J_{\Phi^{-1}\optT}(y^{(k)}) - y^{(k)}}^2_\pspace \leq \norm{y^{(0)} - y^*}^2_\pspace$. 
Since $\sum_{k \in \nset{}{}}\gamma^{(k)}(1 - \gamma^{(k)}) = +\infty$, the difference satisfy $\liminf_{k \to \infty} \norm{J_{\Phi^{-1}\optT}(y^{(k)}) - y^{(k)}}^2_\pspace = 0$, which implies that there exists a subsequence $(y^{(k_i)})_{i \in \nset{}{}}$ such that $\lim_{i \to \infty} \norm{J_{\Phi^{-1}\optT}(y^{(k_i)}) - y^{(k_i)}}^2_\pspace = 0$. 

Moreover, the above subsequence $(y^{(k_i)})_{i \in \nset{}{}}$ is bounded and thus admits a convergent subsubsequence $(y^{(l_i)})_{i \in \nset{}{}}$ with $(l_i)_{i \in \nset{}{}} \subseteq (k_i)_{i \in \nset{}{}}$ such that $\lim_{i \to \infty}y^{(l_i)} = y^\dagger$. 
By the generalization of Implicit Function Theorem \citeApndx{kumagai1980implicit}, it can be proved that $J_{\Phi^{-1}\optT}$ is a continuous mapping. 
Hence, $\lim_{i \to \infty} \norm{J_{\Phi^{-1}\optT}(y^{(l_i)}) - y^{(l_i)}}^2_\pspace = 0$ implies $J_{\Phi^{-1}\optT}(y^\dagger) = y^\dagger$ and $y^\dagger \in \fix{\optT}$, which corresponds to a Nash equilibrium of the original problem by Theorem~\ref{thm:eq-prob}. 
Combining the facts that $\lim_{i \to \infty}\norm{y^{(l_i)} - y^\dagger}^2_\pspace = 0$ and $(\norm{y^{(k)} - y^\dagger}^2_\pspace)_{k \in \nset{}{}}$ is a monotonically decreasing sequence, we can conclude the whole sequence is convergent, i.e., $\lim_{k \in \nset{}{}} y^{(k)} = y^\dagger$. 
\end{proof}

\subsection{Convergence of the Learning Dynamics}\label{appdx:convg-pf}
The proof of Theorem~\ref{thm:main-convg-thm} is largely inspired by that of \citeApndx[Prop.~5.34]{BauschkeHeinzH2017CAaM} for deterministic cases and nonexpansive operators with suitable modifications to our settings. 
Although a similar proof has been included in \citeApndx[Appendix~A]{huang2021distributed} which investigates a distributed solution of stochastic generalized Nash equilibrium problems based on the Douglas-Rachford splitting, we present the following to discuss the convergence conditions for a general continuous quasinonexpansive fixed-point iteration operator. 
\begin{proof}[Proof of Theorem~\ref{thm:main-convg-thm}]

Let $x^* \in \fix{\mathscr{R}_*}$. 
We start by investigating the squared norm $\norm{x^{(k+1)}_* - x^*}^2_\mathcal{H}$ to facilitate our later analysis of $\norm{x^{(k+1)} - x^*}^2_\mathcal{H}$:
\begin{align*}
& \norm{x^{(k+1)}_* - x^*}^2_\mathcal{H} = \norm{(1 - \gamma^{(k)})(x^{(k)} - x^*) + \gamma^{(k)}(\mathscr{R}_*(x^{(k)}) - x^*)}^2_\mathcal{H} \\
&  = (1 - \gamma^{(k)})\norm{x^{(k)} - x^*}^2_\mathcal{H} + \gamma^{(k)}\norm{\mathscr{R}_*(x^{(k)}) - \mathscr{R}_*(x^*)}^2_\mathcal{H} - \gamma^{(k)}(1 - \gamma^{(k)})(\res{x^{(k)}})^2 \\ 
& \leq \norm{x^{(k)} - x^*}^2_\mathcal{H} - \gamma^{(k)}(1 - \gamma^{(k)})(\res{x^{(k)}})^2.
\end{align*}
where the inequality follows from the fact that $\mathscr{R}_*$ is quasinonexpansive. 
Next, we derive a recursive relationship for $\norm{x^{(k+1)} - x^*}^2_\mathcal{H}$ as follows:
\begin{align*}
& \norm{x^{(k+1)} - x^*}^2_\mathcal{H} = \norm{x^{(k+1)} - x^{(k+1)}_* + x^{(k+1)}_* -  x^*}^2_\mathcal{H}\\
& = \norm{\gamma^{(k)}\epsilon^{(k)} + x^{(k+1)}_* - x^*}^2_\mathcal{H} \\
& = \norm{x^{(k+1)}_* - x^*}^2_\mathcal{H} + 2\langle \gamma^{(k)}\epsilon^{(k)}, x^{(k+1)}_* - x^*\rangle_\mathcal{H} + (\gamma^{(k)}\varepsilon^{(k)})^2 \\
& \leq \norm{x^{(k)} - x^*}^2_\mathcal{H} - \gamma^{(k)}(1 - \gamma^{(k)})(\res{x^{(k)}})^2 + 2\gamma^{(k)}\varepsilon^{(k)}\norm{x^{(k)} - x^*}_\mathcal{H} + (\gamma^{(k)}\varepsilon^{(k)})^2,
\end{align*}
where the last inequality follows from the results above and the Cauchy-Schwarz inequality. 
Taking conditional expectation $\expt{}{\cdot \mid \mathcal{F}_{k}}$ on both sides of the above inequality yields:
\begin{align}\label{eq:main-covg-rs-recur}
\begin{split}
& \expt{}{\norm{x^{(k+1)} - x^*}^2_\mathcal{H} \mid \mathcal{F}_{k}} \leq \norm{x^{(k)} - x^*}^2_\mathcal{H} - \gamma^{(k)}(1 - \gamma^{(k)})(\res{x^{(k)}})^2 \\
& \qquad + \expt{}{2\gamma^{(k)}\varepsilon^{(k)}\norm{x^{(k)} - x^*}_\mathcal{H} + (\gamma^{(k)}\varepsilon^{(k)})^2\mid \mathcal{F}_{k}}. 
\end{split}
\end{align}
From the fact that $\sum_{k \in \nset{}{}} \gamma^{(k)}\expt{}{\varepsilon^{(k)} \mid \mathcal{F}_{k}} < +\infty$ and $(\norm{x^{(k)} - x^*}_\mathcal{H})_{k \in \nset{}{}}$ is bounded on a set $\hat{\Omega}$ which has probability one, for any sample $\hat{\omega} \in \hat{\Omega}$, $\sum_{k \in \nset{}{}} \expt{}{2\gamma^{(k)}\varepsilon^{(k)}\norm{x^{(k)} - x^*}_\mathcal{H} + (\gamma^{(k)}\varepsilon^{(k)})^2\mid \mathcal{F}_{k}}(\hat{\omega}) < \infty$. 
By applying the Robbins-Siegmund theorem (\citeApndx{robbins1971convergence}), we can then conclude 
$\sum_{k \in \nset{}{}}\gamma^{(k)}(1 - \gamma^{(k)})(\res{x^{(k)}})^2 < \infty$. 
Combining this with the condition that $\sum_{k \in \nset{}{}}\gamma^{(k)}(1 - \gamma^{(k)}) = +\infty$ with $\gamma^{(k)} \in [0, 1]$ yields 
$\liminf_{k \to \infty} (\res{x^{(k)}})^2 = 0$, i.e., 
there exists a subsequence $(x^{(k_i)})_{i \in \nset{}{}}$, such that $\lim_{i \to \infty} \res{x^{(k_i)}} = 0$. 

Moreover, by assumption, $(x^{(k)})_{k \in \nset{}{}}$ is bounded sequence on the set $\hat{\Omega}$, and thus its subsequence $(x^{(k_i)})_{i \in \nset{}{}}$ admits a convergent subsubsequence $(x^{(l_i)})_{i \in \nset{}{}}$ such that $\lim_{i \to \infty} x^{(l_i)} = x^\dagger$, where $(l_i)_{i \in \nset{}{}} \subseteq (k_i)_{i \in \nset{}{}}$. 
The relation $(\res{x^{\dagger}})^2 = 0$ implies $x^\dagger \in \fix{\mathscr{R}_*}$, which follows from the continuity of $\mathscr{R}_*$. 
We can substitute $x^*$ in \eqref{eq:main-covg-rs-recur} with $x^\dagger$. 
By the Robbins-Siegmund theorem, $\lim_{k \to \infty} \norm{x^{(k)} - x^\dagger}^2_\pspace$ exists a.s. 
Since $(x^{(l_i)})_{i \in \nset{}{}}$ is a subsequence of $(x^{(k)})_{k \in \nset{}{}}$ converging to the fixed point $x^\dagger$, we can conclude that $\lim_{k \to \infty} \norm{x^{(k)} - x^\dagger}^2_\pspace = 0$, and hence $\lim_{k \to \infty}x^{(k)} = x^\dagger$.
\end{proof}

To distinguish the computation results of $\mathscr{R}^{(k)}$ and $\mathscr{R}^*$, let $\tilde{y}^{(k+1)} \coloneqq \mathscr{R}^{(k)}(y^{(k)})$ and $\tilde{y}^{(k+1)}_{*} \coloneqq \mathscr{R}_*(y^{(k)})$. 
As a reminder, we define $\norm{\Delta \hat{w}^{(k)}_i}_2 \coloneqq \norm{\hat{w}^{(k)}_i - w^*_i}_2$, and $\norm{\Delta \hat{w}^{(k)}}_2 \coloneqq \norm{\hat{w}^{(k)} - w^*}_2$.
\begin{lemma}\label{le:resol-upper-bd}
Suppose Assumptions~\ref{asp:commtopo} to \ref{asp:olse} hold. There exist positive constants $\tilde{\alpha}$ and $\tilde{\beta}$ such that at each $k$-th iteration, the estimation error $\varepsilon^{(k)}$ are upper bounded as follows:
\begin{equation*}
\expt{}{\varepsilon^{(k)} \mid \mathcal{F}_k} \leq (\tilde{\alpha}\norm{y^{(k)}}_\pspace + \tilde{\beta}) \expt{}{\norm{\Delta\hat{w}^{(k)}}_2 \mid \mathcal{F}_k}. 
\end{equation*}
\end{lemma}
\begin{proof}
For each player $i \in \playerN$, let $\tilde{y}^{i(k+1)}_{i*}$ and $\tilde{y}^{i(k+1)}_{i}$ be the $\argmin$ solutions of the augmented objective $\hat{\mathbb{J}}^{(k)}_i(\tilde{y}^i_i; w^*_i)$ and its estimate $\hat{\mathbb{J}}^{(k)}_i(\tilde{y}^i_i; \hat{w}^{(k)}_i)$ within $\mathcal{X}_i$, respectively. 
Based on Assumption~\ref{asp:obj}, $\tilde{y}^{i(k+1)}_{i*}$ and $\tilde{y}^{i(k+1)}_{i}$ should satisfy the following two inclusions: 
\begin{align}
\begin{split}
& \partial_{x_i}\mathbbm{f}_i(\tilde{y}^{i(k+1)}_i) + \nabla_{x_i}\mathbbm{g}_i(\tilde{y}^{i(k+1)}_i; \hat{w}^{(k)}_i) + \frac{1}{\tau_{i0}}(\tilde{y}^{i(k+1)}_i - y^{i(k)}_i) + N_{\mathcal{X}_i}(\tilde{y}^{i(k+1)}_i) \ni \bzero, \\
& \partial_{x_i}\mathbbm{f}_i(\tilde{y}^{i(k+1)}_{i*}) + \nabla_{x_i}\mathbbm{g}_i(\tilde{y}^{i(k+1)}_{i*}; w^*_i) + \frac{1}{\tau_{i0}}(\tilde{y}^{i(k+1)}_{i*} - {y}^{i(k)}_i) + N_{\mathcal{X}_i}(\tilde{y}^{i(k+1)}_{i*}) \ni \bzero, 
\end{split}
\end{align}
where we omit $\tilde{y}^{+(k+1)}_i$ in the inputs of $\mathbbm{f}_i$ and $\mathbbm{g}_i$. 
Deducting one inclusion from the other and reformulating the obtained inclusion, we then have:
\begin{align}\label{eq:upperbound-inclu}
\begin{split}
& \partial_{x_i}\mathbbm{f}_i(\tilde{y}^{i(k+1)}_i) - \partial_{x_i}\mathbbm{f}_i(\tilde{y}^{i(k+1)}_{i*})
+ \nabla_{x_i}\mathbbm{g}_i(\tilde{y}^{i(k+1)}_{i}; w^*_i) - \nabla_{x_i}\mathbbm{g}_i(\tilde{y}^{i(k+1)}_{i*}; w^*_i)   \\
& \qquad + \frac{1}{\tau_{i0}}(\tilde{y}^{i(k+1)}_i - \tilde{y}^{i(k+1)}_{i*}) + N_{\mathcal{X}_i}(\tilde{y}^{i(k+1)}_i) - N_{\mathcal{X}_i}(\tilde{y}^{i(k+1)}_{i*})  \\
&  \ni \nabla_{x_i}\mathbbm{g}_i(\tilde{y}^{i(k+1)}_{i}; w^*_i) - \nabla_{x_i}\mathbbm{g}_i(\tilde{y}^{i(k+1)}_i; \hat{w}^{(k)}_i).
\end{split}
\end{align}
We then take the inner product of both sides of the above inclusion with $\tilde{y}^{i(k+1)}_i - \tilde{y}^{i(k+1)}_{i*}$. 
By the convexity suggested in Assumption~\ref{asp:obj} and Assumption~\ref{asp:feaset}, the left-hand side enjoys the strong monotonicity as follows: 
\begin{align}\label{eq:upperbound-lhs}
\begin{split}
& \langle \partial_{x_i}\mathbbm{f}_i(\tilde{y}^{i(k+1)}_i) - \partial_{x_i}\mathbbm{f}_i(\tilde{y}^{i(k+1)}_{i*}) + \nabla_{x_i}\mathbbm{g}_i(\tilde{y}^{i(k+1)}_{i}; w^*_i) - \nabla_{x_i}\mathbbm{g}_i(\tilde{y}^{i(k+1)}_{i*}; w^*_i)\\
& \qquad + \frac{1}{\tau_{i0}}(\tilde{y}^{i(k+1)}_i - \tilde{y}^{i(k+1)}_{i*}) + N_{\mathcal{X}_i}(\tilde{y}^{i(k+1)}_i) - N_{\mathcal{X}_i}(\tilde{y}^{i(k+1)}_{i*}) , \tilde{y}^{i(k+1)}_i - \tilde{y}^{i(k+1)}_{i*}\rangle \\
&  \geq \frac{1}{\tau_{i0}}\norm{\tilde{y}^{i(k+1)}_i - \tilde{y}^{i(k+1)}_{i*}}^2_2. 
\end{split}
\end{align}
By the Cauchy-Schwarz inequality and the Lipschitz continuity of $\nabla_{x_i}\mathbbm{g}_i$ given in Assumption~\ref{asp:obj}, taking the inner product of the right-hand side of \eqref{eq:upperbound-inclu} with $\tilde{y}^{i(k+1)}_i - \tilde{y}^{i(k+1)}_{i*}$ yields:
\begin{align}\label{eq:upperbound-rhs}
\begin{split}
& \langle \nabla_{x_i}\mathbbm{g}_i(\tilde{y}^{i(k+1)}_{i}; w^*_i) - \nabla_{x_i}\mathbbm{g}_i(\tilde{y}^{i(k+1)}_i; \hat{w}^{(k)}_i), \tilde{y}^{i(k+1)}_i - \tilde{y}^{i(k+1)}_{i*}\rangle \\
& \leq (\alpha_{g,i} \norm{\tilde{y}^{+(k+1)}_i}_2 + \beta_{g,i}) \norm{\hat{w}^{(k)}_i - w^*_i}_2 \norm{\tilde{y}^{i(k+1)}_i - \tilde{y}^{i(k+1)}_{i*}}_2.
\end{split}
\end{align}
Combining \eqref{eq:upperbound-lhs} and \eqref{eq:upperbound-rhs} yields
\begin{align*}
\begin{split}
\norm{\tilde{y}^{i(k+1)}_i - \tilde{y}^{i(k+1)}_{i*}}_2 &\leq \tau_{i0}(\alpha_{g,i}\norm{\tilde{y}^{+(k+1)}_i}_2 + \beta_{g,i})\norm{\Delta\hat{w}^{(k)}_i}_2 \\
& \leq (\alpha_{i}\norm{y^{(k)}}_2 + \beta_{i})\norm{\Delta\hat{w}^{(k)}_i}_2, 
\end{split}
\end{align*}
where the existence of the positive constants $\alpha_i$ and $\beta_i$ follows from the fact that $\tilde{y}^{+(k+1)}_i$ is some linear transformation of $y^{(k)}$. 
Subsequently, stacking the difference of all players, we obtain:
\begin{align*}
\begin{split}
\norm{\tilde{y}^{(k+1)} - \tilde{y}^{(k+1)}_*}_2 &\leq (\sum_{i \in \playerN}(\alpha_{i}\norm{y^{(k)}}_2 + \beta_{i})^2\norm{\Delta\hat{w}^{(k)}_i}^2_2)^{1/2} \\
& \leq (\bar{\alpha}\norm{y^{(k)}}_2 + \bar{\beta})(\sum_{i \in \playerN} \norm{\Delta\hat{w}^{(k)}_i}^2_2)^{1/2} \\
& = (\bar{\alpha}\norm{y^{(k)}}_2 + \bar{\beta})\norm{\Delta \hat{w}^{(k)}}_2,
\end{split}
\end{align*}
where $\bar{\alpha} \coloneqq \max\{\alpha_i: i \in \playerN\}$, and $\bar{\beta} \coloneqq \max\{\beta_i: i \in \playerN\}$. 
Hence, taking conditional expectation $\expt{}{\cdot \mid \mathcal{F}_k}$ on both sides yields: $\expt{}{\norm{\tilde{y}^{(k+1)} - \tilde{y}^{(k+1)}_*}_2 \mid \mathcal{F}_k} \leq (\bar{\alpha}\norm{y^{(k)}}_2 + \bar{\beta})\expt{}{\norm{\Delta\hat{w}^{(k)}_i}_2 \mid \mathcal{F}_k}$. 
We close this proof by converting the above relation from Euclidean space to the inner-product space $\pspace$. 
The positive definite design matrix $\Phi$ has its maximum (resp. minimum) eigenvalue denoted by $\bar{\sigma}_{\Phi}$ (resp. $\ubar{\sigma}_{\Phi}$), and we have:
\begin{align*}
\begin{split}
\expt{}{\norm{\tilde{y}^{(k+1)} - \tilde{y}^{(k+1)}_*}_\pspace \mid \mathcal{F}_k} &\leq \sqrt{\bar{\sigma}_{\Phi}}(\frac{\bar{\alpha}}{\sqrt{\ubar{\sigma}_{\Phi}}}\norm{y^{(k)}}_\pspace + \bar{\beta})\expt{}{\norm{\Delta\hat{w}^{(k)}_i}_2 \mid \mathcal{F}_k} \\
& \leq (\tilde{\alpha} \norm{y^{(k)}}_\pspace + \tilde{\beta})\expt{}{\norm{\Delta\hat{w}^{(k)}_i}_2 \mid \mathcal{F}_k},
\end{split}
\end{align*}
where $\tilde{\alpha} \coloneqq \bar{\alpha}\sqrt{\bar{\sigma}_{\Phi}/\ubar{\sigma}_{\Phi}}$, and $\tilde{\beta} \coloneqq \bar{\beta}\sqrt{\bar{\sigma_\Phi}}$. 
\end{proof}

\begin{proof}[Proof of Theorem~\ref{thm:bdd-est-seq-sumb}]
Define $y^{(k+1)}_* \coloneqq \mathscr{P}_*(y^{(k)})$ and $y^{(k+1)} \coloneqq \mathscr{P}^{(k)}(y^{(k)})$.
To leverage the convergence result in Theorem~\ref{thm:main-convg-thm}, note that 
\begin{align*}
\begin{split}
\expt{}{\norm{y^{(k+1)} - y^*}_\pspace \mid \mathcal{F}_k} &= \expt{}{\norm{y^{(k+1)} - y^{(k+1)}_* + y^{(k+1)}_* - y^*}_\pspace \mid \mathcal{F}_k} \\
& \leq \gamma^{(k)}\expt{}{\varepsilon^{(k)} \mid \mathcal{F}_k} + \expt{}{\mathscr{P}_*(y^{(k)}) - \mathscr{P}_*(y^*) \mid \mathcal{F}_k}. 
\end{split}
\end{align*}
We can further simplify the above relation by applying Lemma~\ref{le:resol-upper-bd} and the quasinonexpansiveness of $\mathscr{P}_*$ and obtain the following results (a.s.):
\begin{align*}
& \expt{}{\norm{y^{(k+1)} - y^*}_\pspace \mid \mathcal{F}_k} \leq  \gamma^{(k)}\expt{}{\norm{\Delta\hat{w}^{(k)}}_2 \mid \mathcal{F}_k}(\tilde{\alpha} \norm{y^{(k)}}_\pspace + \tilde{\beta}) + \expt{}{\norm{y^{(k)} - y^*}_\pspace \mid \mathcal{F}_k} \\
& = \gamma^{(k)}\expt{}{\norm{\Delta\hat{w}^{(k)}}_2 \mid \mathcal{F}_k}(\tilde{\alpha} \norm{y^{(k)} - y^* + y^*}_\pspace + \tilde{\beta}) + \norm{y^{(k)} - y^*}_\pspace \\ 
& \leq (1 + \tilde{\alpha} \gamma^{(k)}\expt{}{\norm{\Delta\hat{w}^{(k)}}_2 \mid \mathcal{F}_k})\norm{y^{(k)} - y^*}_\pspace + \gamma^{(k)}\expt{}{\norm{\Delta\hat{w}^{(k)}}_2 \mid \mathcal{F}_k}(\tilde{\alpha}\norm{y^*}_\pspace + \tilde{\beta}).
\end{align*}
Since $\norm{y^*}_\pspace < \infty$ and we assume that $(\gamma^{(k)}\expt{}{\norm{\Delta\hat{w}^{(k)}}_2 \mid \mathcal{F}_k})_{k \in \nset{}{}}$ is a summable sequence, the Robbins-Siegmund theorem (\citeApndx{robbins1971convergence}) can be applied to show $\lim_{k \to \infty} \norm{y^{(k)} - y^*}_\pspace$ exists and is finite a.s.
Consequently, there exists a set $\hat{\Omega}$ which has probability one, such that for any $\hat{\omega} \in \hat{\Omega}$, the sequence $(\norm{y^{(k)}(\hat{\omega}) - y^*}_\pspace)_{k \in \nset{}{}}$ is bounded. 
Therefore, we can find some constant $B(\hat{\omega})$ which satisfies, for all $k \in \nset{}{}$, $\norm{y^{(k)}(\hat{\omega})}_\pspace = \norm{y^{(k)}(\hat{\omega}) - y^* + y^*}_\pspace \leq \norm{y^{(k)}(\hat{\omega}) - y^*}_\pspace + \norm{y^*}_\pspace \leq B(\hat{\omega})$. 

Since the deterministic sequence $(\norm{y^{(k)}(\hat{\omega})}_\pspace)_{k \in \nset{}{}}$ is upper bounded by a constant $B(\hat{\omega})$ for any $\hat{\omega} \in \hat{\Omega}$, 
combining Lemma~\ref{le:resol-upper-bd} and the summability of $(\gamma^{(k)}\expt{}{\norm{\Delta\hat{w}^{(k)}}_2 \mid \mathcal{F}_k})_{k \in \nset{}{}}$, 
we finally can conclude 
\begin{align*}
& \sum_{k \in \nset{}{}}\gamma^{(k)}\expt{}{\varepsilon^{(k)} \mid \mathcal{F}_k}(\hat{\omega}) \\
& \leq \sum_{k \in \nset{}{}}\gamma^{(k)}\expt{}{\norm{\Delta\hat{w}^{(k)}}_2 \mid \mathcal{F}_k} (\tilde{\alpha} \norm{y^{(k)}(\hat{\omega})}_\pspace + \tilde{\beta}) \\
& \leq \sum_{k \in \nset{}{}}\gamma^{(k)} \expt{}{\norm{\Delta\hat{w}^{(k)}}_2 \mid \mathcal{F}_k}(\tilde{\alpha} B(\hat{\omega}) + \tilde{\beta}) < \infty \text{ a.s.}
\end{align*}
\end{proof}

\subsection{Convergence Analysis of the OLSE}\label{appdx:olse-pf}

Our goal in this appendix is to prove Theorem~\ref{thm:obsv-num-est-err} concerning the asymptotic convergence rate result of the OLSE learning dynamics described in Subroutine~\ref{alg:param-est}. 
The proof will be based on the following law of large deviations for multi-dimensional M-estimators.  
We first define the noise-free neighboring aggregate function as $\bar{s}_i(\check{y}^{+(k)}_i; \hat{w}_i) \coloneqq \hat{w}_{ii} + \sum_{j \in \neighbN{+}{i}}\hat{w}^{T}_{ji}\check{y}^{j(k)}_j = \langle [1; (\check{y}^{+(k)}_i)^T], \hat{w}_i\rangle$, where $\check{y}^{+(k)}_i \coloneqq [\check{y}^{j(k)}_j]_{j \in \neighbN{+}{i}}$.
In addition, it is worth mentioning here that the parameter learning dynamics only involves local decisions $\{y^i_i\}_{i \in \playerN}$, the feasibility regions of which are bounded by Assumption~\ref{asp:feaset}. 

\begin{theorem}\label{thm:lld}\citeApndx[Thm.~3.2]{sieders1987large}
For each player $i \in \playerN$, suppose for all $k \in \nset{}{}$ and $\abs{b} \leq +\infty$, there exist some $\kappa > 0$, such that $\expt{}{\exp(b \xi^{(k)}_{i})} \leq \exp(\frac{1}{2}\kappa b^2)$. 
Moreover, let there exist positive constants $\ubar{D}_{i}$ and $\bar{D}_{i}$ such that, for all $\hat{w}_i, \hat{w}_i' \in \mathcal{W}_i$ and $K$ sufficiently large, 
\begin{equation}\label{eq:lld-cond}
\ubar{D}_{i}\norm{\sqrt{K} (\hat{w}_i - \hat{w}_i')}^2 \leq \sum_{1 \leq k \leq K} \big( \bar{s}_i(\check{y}^{+(k)}_i; \hat{w}_i) - \bar{s}_i(\check{y}^{+(k)}_i; \hat{w}_i') \big)^2 \leq \bar{D}_{i}\norm{\sqrt{K} (\hat{w}_i - \hat{w}_i')}^2.
\end{equation}
Then the following law of large deviation holds for the least square estimator $\hat{w}^{(K)}_{i}$: there exist positive constants $D'_{i}$ and $D^*_{i}$ such that, for all $K$ and $H$ large enough,
\begin{equation}
\sup_{\hat{w}_i \in \mathcal{W}_i} \mathbb{P}_{\hat{w}_i} \big\{ \norm{\sqrt{K}(\hat{w}^{(K)}_{i} - \hat{w}_i)}_2 \geq H\big\} \leq D'_{i}\exp(-D^*_{i}H^2). 
\end{equation}
\end{theorem}

Next, we are going to show that the sequence generated by Algorithm~\ref{alg:assemble} fulfills the requirements in Theorem~\ref{thm:lld}. 
Given the function $\bar{s}_i$ and the compactness of $\mathcal{X}$, it is straightforward to conclude the existence of an upper bound $\bar{D}_{i}$. 
In the rest of this section, we will focus on proving the existence of the lower bound $\ubar{D}_{i}$ under the setup of Subroutine~\ref{alg:param-est}. 

We begin with proving a lemma concerning the properties of the product of two independent random variables based on the law of large numbers for martingale difference \citeApndx[Thm.~2.18]{hall2014martingale}. 

\begin{lemma}\label{le:convg-mtg}
Suppose $(X^{(t)})_{t \in \nset{}{}}$ is a sequence of random variables with bounded range, $(Y^{(t)})_{t \in \nset{}{}}$ is a sequence of random variables that are zero-mean, range-bounded, and independent and identically distributed. 
In addition, $(X^{(t)})_{t \in \nset{}{}}$ and $(Y^{(t)})_{t \in \nset{}{}}$ are independent. 
Define the random variable $Z^{(t)} \coloneqq X^{(t)} Y^{(t)}$.
Then $\lim_{k \to \infty} \frac{1}{k+1}\sum_{t=0}^{k}Z^{(t)} = 0$ a.s.
\end{lemma}

\begin{proof}
Construct the $\sigma-$field $\mathcal{F}_k \coloneqq \{X^{(0)}, Y^{(0)}, \ldots, X^{(k)}, Y^{(k)}\}$ and $(\mathcal{F}_k)_{k \in \nset{}{}}$ is a filtration. 
Define $S^{(k)} = \sum_{t=0}^{k}Z^{(t)}$. 
We claim that $(S^{(k)})_{k \in \nset{}{}}$ is a martingale w.r.t. the filtration $(\mathcal{F}_k)_{k \in \nset{}{}}$. 
For each $k \in \nset{}{}$, $S^{(k)}$ is measurable w.r.t. the filtration $\mathcal{F}_k$, and $\expt{}{S^{(k)}} = \expt{}{\sum_{t=0}^{k}X^{(t)} Y^{(t)}} = \sum_{t=0}^{k}\expt{}{X^{(t)}}\expt{}{Y^{(t)}} = 0 < \infty$ by the independence of $X^{(k)}$ and $Y^{(k)}$. 
Moreover, for all $k_1, k_2 \in \nset{}{}$ and $k_1 < k_2$,
\begin{align*}
\expt{}{S^{(k_2)} \mid \mathcal{F}_{k_1}} &= \expt{}{S^{(k_1)} + \sum_{t = k_1 + 1}^{k_2}Z^{(t)} \mid \mathcal{F}_{k_1}}\\
& = \expt{}{S^{(k_1)} \mid \mathcal{F}_{k_1}} + \sum_{t=k_1 + 1}^{k_2} \expt{}{Z^{(t)} \mid \mathcal{F}_{k_1}} \\
& = S^{(k_1)} + \sum_{t=k_1 + 1}^{k_2} \expt{}{Z^{(t)} \mid \mathcal{F}_{k_1}} \\
& = S^{(k_2)} + \sum_{t=k_1 + 1}^{k_2} \expt{}{X^{(t)} \mid \mathcal{F}_{k_1}}\expt{}{Y^{(t)}} = S^{(k_1)}, \\
\end{align*}
which completes the proof that $(S^{(k)})_{k \in \nset{}{}}$ is a martingale w.r.t. the filtration $(\mathcal{F}_k)_{k \in \nset{}{}}$. 

Next, since for each $t \in \nset{}{}$, $Z^{(t)}$ is a random variable with bounded range, there exists a constant upper bound $B^2_Z$ such that $\expt{}{\abs{Z^{(t)}}^2 \mid \mathcal{F}_{t-1}} \leq B^2_Z$. 
Building upon this upper bound, we have 
\begin{align*}
\sum_{t=0}^{\infty} \frac{1}{(t + 1)^2}\expt{}{\abs{Z^{(t)}}^2 \mid \mathcal{F}_{t-1}}
\leq \sum_{t=0}^{\infty} \frac{1}{(t + 1)^2}B^2_Z < \infty.
\end{align*}

Thus, by the law of large numbers for martingale difference \citeApndx[Thm.~2.18]{hall2014martingale}, we obtain $\lim_{k \to \infty}\frac{1}{k+1}S^{(k)} = 0$ a.s.
\end{proof}

With this lemma at hand, we can construct the following uniform lower bound to satisfy the identifiability condition in Theorem~\ref{thm:lld}. 
Before proceeding, we make some notational conventions for the sake of readability: $\delta^{+(k)}_i \coloneqq [\delta^{(k)}_j]_{j \in \neighbN{+}{i}}$, $y^{+(k)}_{i\delta} \coloneqq [y^{j(k)}_{j\delta}]_{j \in \neighbN{+}{i}}$, and $\mathcal{X}^+_i \coloneqq \prod_{j \in \neighbN{+}{i}} \mathcal{X}_j$. 
Under Assumption~\ref{asp:feaset}, let $\bar{X}^+_i$ denote a constant satisfying $\norm{x^+_i}_\infty \leq \bar{X}^+_i$ for any stack decision vector $x^+_i \in \mathcal{X}^+_i$, and $C_i = \bar{X}^+_i\sqrt{n^+_i}$. 
\begin{lemma}\label{le:lower-bd}
Suppose Assumptions~\ref{asp:feaset} and \ref{asp:rdm-explr} hold. 
Moreover, suppose $\expt{}{\delta^+_i(\delta^+_i)^T} - \bar{\sigma}^+_iI_{n^+_i} \succcurlyeq 0$ for some positive $\bar{\sigma}^+_i \leq C_i$. 
Define the constant $\ubar{D}_i \coloneqq \frac{1}{4}\bar{\sigma}^+_i\min\{C_i^{-2}, 1\}$. 
Then on a sample set $\hat{\Omega}$ which has probability one, for any sample $\hat{\omega} \in \hat{\Omega}$, we can find a $K_{\text{id}}(\hat{\omega})$ sufficiently large such that for all $K \geq K_{\text{id}}(\hat{\omega})$ and $\hat{w}_i, \hat{w}_i' \in \mathcal{W}_i$, the following inequality holds a.s.:
\begin{equation}\label{eq:iden-target}
    \frac{1}{K}\sum_{k=1}^{K} \big( \bar{s}_i(\check{y}^{+(k)}_i; \hat{w}_i) - \bar{s}_i(\check{y}^{+(k)}_i; \hat{w}_i') \big)^2 \geq \ubar{D}_{i}\norm{\hat{w}_i - \hat{w}_i'}^2. 
\end{equation}
\end{lemma}
\begin{proof}
Converting the L.H.S. of \eqref{eq:iden-target} into a matrix form, we have:
\begin{align*}
\begin{split}
& \frac{1}{K} \sum_{k=1}^{K} (\bar{s}_i(\check{y}^{+(k)}_i; \hat{w}_i) - \bar{s}_i(\check{y}^{+(k)}_i; \hat{w}_i'))^2 = \frac{1}{K} \sum_{k=1}^{K} (\langle [1, (\check{y}^{+(k)}_i)^T], \hat{w}_i\rangle - \langle [1, (\check{y}^{+(k)}_i)^T], \hat{w}_i'\rangle)^2 \\
& \qquad = \frac{1}{K} \sum_{k=1}^{K} (\hat{w}_i - \hat{w}_i')^T [1; \check{y}^{+(k)}_i] \cdot [1, (\check{y}^{+(k)}_i)^T](\hat{w}_i - \hat{w}_i') \\
& \qquad = (\hat{w}_i - \hat{w}_i')^T \Big(\frac{1}{K}\sum_{k=1}^{K}[1; \check{y}^{+(k)}_i] \cdot [1, (\check{y}^{+(k)}_i)^T]\Big) (\hat{w}_i - \hat{w}_i') \\
& \qquad = (\hat{w}_i - \hat{w}_i')^T \Big(\frac{1}{K}\sum_{k=1}^{K} 
\begin{bmatrix}
1 & (\check{y}^{+(k)}_i)^T\\
\check{y}^{+(k)}_i & \check{y}^{+(k)}_i(\check{y}^{+(k)}_i)^T
\end{bmatrix}
\Big) (\hat{w}_i - \hat{w}_i')\\
& \qquad = (\hat{w}_i - \hat{w}_i')^T
\begin{bmatrix}
1 & (\check{\by}_K)^T \\
\check{\by}_K & Y_K
\end{bmatrix}
(\hat{w}_i - \hat{w}_i'),
\end{split}
\end{align*}
where, in the last line, we define $\check{\by}_K \coloneqq \frac{1}{K}\sum_{k=1}^{K}\check{y}^{+(k)}_i$, and $Y_K \coloneqq \frac{1}{K}\sum_{k=1}^{K}\check{y}^{+(k)}_i(\check{y}^{+(k)}_i)^T$. 

We next investigate the asymptotic behavior of the matrix $Y_K$. 
Recall that $\check{y}^{+(k+1)}_i \coloneqq {y}^{+(k+1)}_{i\delta} + \delta^{+(k)}_i$ from Subroutine~\ref{alg:param-est}, and $Y_K$ can be expressed as: 
\begin{align*}
Y_K \coloneqq \frac{1}{K}\sum_{k=1}^{K} \Big(y^{+(k)}_{i\delta}(y^{+(k)}_{i\delta})^T +  y^{+(k)}_{i\delta}(\delta^{+(k-1)}_i)^T + \delta^{+(k-1)}_i(y^{+(k)}_{i\delta})^T + \delta^{+(k-1)}_i(\delta^{+(k-1)}_i)^T\Big). 
\end{align*}
Note that the first term $\frac{1}{K}\sum_{k=1}^{K} y^{+(k)}_{i\delta}(y^{+(k)}_{i\delta})^T$ is a positive semi-definite matrix, and the last term $\lim_{K \to \infty} \frac{1}{K}\sum_{k=0}^{K-1}\delta^{+(k)}_i(\delta^{+(k)}_i)^T = \expt{}{\delta^{+(k)}_i(\delta^{+(k)}_i)^T} \succcurlyeq \bar{\sigma}^+_i I_{n^+_i}$ a.s., by the strong law of large numbers. 
For the remaining two terms, we can apply Lemma~\ref{le:convg-mtg} to show that 
\begin{align*}
\frac{1}{K}\sum_{k=1}^{K}y^{+(k)}_{i\delta}(\delta^{+(k-1)}_i)^T = \bzero_{n^+_i \times n^+_i} \; \text{a.s.},
\end{align*}
and similarly for its transpose. 

Define the matrix $\ubar{Y} \coloneqq \frac{1}{K}\sum_{k=1}^{K}y^{+(k)}_{i\delta}(y^{+(k)}_{i\delta})^T + \frac{3}{4} \bar{\sigma}^+_iI_{n^+_i} \succ 0$. 
Altogether, on a sample set $\hat{\Omega}$ which has probability one, for any arbitrary sample $\hat{\omega} \in \hat{\Omega}$, there exists a time index $K(\hat{\omega})$, such that for any $K \geq K(\hat{\omega})$, $Y_{K} \succ \ubar{Y}$. 

Now, to show that $\Big[\begin{smallmatrix} 1 & (\check{\by}_K)^T \\ \check{\by}_K & Y_K \end{smallmatrix}\Big] \succ 0$, it suffices to prove that $\Big[\begin{smallmatrix} 1 & (\check{\by}_K)^T \\ \check{\by}_K & \ubar{Y} \end{smallmatrix}\Big] \succ 0$. 
By the Schur complement, the above relation holds if and only if $\ubar{Y} - \check{\by}_K(\check{\by}_K)^T \succ 0$. 
Let $\by_{\delta K} \coloneqq \frac{1}{K}\sum_{k=1}^{K}y^{+(k)}_{i\delta}$, and ${\boldsymbol \delta}_K \coloneqq \frac{1}{K}\sum_{k=0}^{K-1} \delta^{+(k)}_{i}$. 
Then $\check{\by}_K(\check{\by}_K)^T = \by_{\delta K}(\by_{\delta K})^T + R({\boldsymbol \delta}_K)$, where $R({\boldsymbol \delta}_K) \coloneqq \by_{\delta K}({\boldsymbol \delta}_K)^T + {\boldsymbol \delta}_K(\by_{\delta K})^T + {\boldsymbol \delta}_K({\boldsymbol \delta}_K)^T$. 
Again by the strong law of large numbers, ${\boldsymbol \delta}_K \to \bzero$ a.s. as $K \to \infty$, which further implies $\frac{1}{4}\bar{\sigma}^+_iI_{n^+_i} - R({\boldsymbol \delta}_K) \succcurlyeq 0$ when $K \geq \Tilde{K}(\hat{\omega})$ for some sufficiently large $\Tilde{K}(\hat{\omega}) \geq K(\hat{\omega})$ for $\hat{\omega}$ in a subset of $\hat{\Omega}$ still with probability one. 
Moreover, the remaining terms in $\bar{Y} - \check{\boldsymbol{y}}_K(\check{\boldsymbol{y}}_K)^T$ give:
\begin{align}\label{eq:yidelta-ineq}
\begin{split}
& \frac{1}{K}\sum_{k=1}^{K} y^{+(k)}_{i\delta}(y^{+(k)}_{i\delta})^T - \by_{\delta K}(\by_{\delta K})^T \\
& \qquad = \frac{1}{K^2}\Big( \sum_{k=1}^K\sum_{t=k+1}^{K} (y^{+(k)}_{i\delta} - y^{+(t)}_{i\delta})(y^{+(k)}_{i\delta} - y^{+(t)}_{i\delta})^T \Big) \succcurlyeq 0,
\end{split}
\end{align}
which completes the proof of $\ubar{Y} - \check{\by}_K(\check{\by}_K)^T \succ 0$. 

Likewise, by applying similar arguments as above, we are going to show:
\begin{align*}
\begin{bmatrix} 1 & (\check{\by}_K)^T \\ \check{\by}_K & \ubar{Y} \end{bmatrix}
- \begin{bmatrix} \frac{1}{4}\bar{\sigma}^+_i C_i^{-2} & 0 \\ 0 & \frac{1}{4}\bar{\sigma}^+_i I_{n^+_i} \end{bmatrix} \succ 0,
\end{align*}
with $\bar{\sigma}^+_i C_i^{-2} \leq 1$ by assumption. 
Then, it suffices for us to prove the following:
\begin{align*}
    (\ubar{Y} - \frac{1}{4}\bar{\sigma}^+_i I_{n^+_i}) - \frac{1}{1 - \frac{1}{4}\bar{\sigma}^+_i C_i^{-2}}\check{\by}_K(\check{\by}_K)^T \succ 0. 
\end{align*}
As suggested above, $R({\boldsymbol \delta}_K) \to \bzero$ a.s. as $K \to \infty$. 
Then there exists some sufficiently large $\tsup{K}(\hat{\omega}) \geq \tilde{K}(\hat{\omega})$, such that for any $K \geq \tsup{K}(\hat{\omega})$, this LMI can be further bounded from below by the following:
\begin{align*}
\begin{split}
& (\ubar{Y} - \frac{1}{4}\bar{\sigma}^+_i I_{n^+_i}) - (1 + \frac{\frac{1}{4}\bar{\sigma}^+_i C_i^{-2}}{1 - \frac{1}{4}\bar{\sigma}^+_i C_i^{-2}})\check{\by}_K(\check{\by}_K)^T \\
& \overset{(a)}{\succcurlyeq} (\ubar{Y} - \frac{1}{4}\bar{\sigma}^+_i I_{n^+_i}) - (1 + \frac{1}{3}\bar{\sigma}^+_i C_i^{-2})\check{\by}_K(\check{\by}_K)^T \\
& \overset{(b)}{\succcurlyeq} \frac{1}{2}\bar{\sigma}^+_i I_{n^+_i} - (1 + \frac{1}{3}\bar{\sigma}^+_i C_i^{-2})R({\boldsymbol \delta}_K) - \frac{1}{3}\bar{\sigma}^+_i C_i^{-2}\by_{\delta K}(\by_{\delta K})^T \\
& \overset{(c)}{\succcurlyeq} \frac{1}{6}\bar{\sigma}^+_i I_{n^+_i} - (1 + \frac{1}{3}\bar{\sigma}^+_i C_i^{-2})R({\boldsymbol \delta}_K) \overset{(d)}{\succ} 0. 
\end{split}
\end{align*}
For (a), we use $\bar{\sigma}^+_i C_i^{-2} \leq 1$; 
for (b), we apply the definition of $\bar{Y}$ and \eqref{eq:yidelta-ineq}; 
(c) follows from the fact that $\norm{\by_{\delta K}}_{\infty} \leq \bar{X}^+_i$ since for each $k$ and $j \in \neighbN{+}{i}$, the adjusted decision $y^{j(k)}_{j\delta} \in \mathcal{X}_j$, and by Gershgorin circle theorem, the matrix $I_{n^+_i} - C^{-2}_i\by_{\delta K}(\by_{\delta K})^T \succcurlyeq 0$; 
finally, (d) is the result of the fact that every entry of $R(\boldsymbol \delta_K)$ converges to zero a.s.
Therefore, we can select $K_{\text{id}}(\hat{\omega}) \geq \tsup{K}(\hat{\omega})$, such that for any $K \geq K_{\text{id}}(\hat{\omega})$, 
\begin{align*}
\begin{bmatrix}
1 & (\check{\by}_K)^T \\ \check{\by}_K & Y_K
\end{bmatrix}  - \ubar{D}_{i} I_{n^+_i+1} \succ 0,
\end{align*}
where the lower bound $\ubar{D}_i \coloneqq \frac{1}{4}\bar{\sigma}^+_i\min\{C_i^{-2}, 1\}$, as claimed. 
\end{proof}

In the following lemma, we are going to establish an upper bound for the distance between the sequential solutions to the OLSE:
\begin{align}\label{eq:olse-linear-copy}
\hat{w}^{(k+1)}_i \coloneqq \argmin_{w_i \in \mathcal{W}_i} \frac{1}{k+1}\sum_{t=0}^{k}
(s^{(t)}_{i} - \langle \ell^{(t)}_{i}, w_i\rangle)^2.
\end{align}

\begin{lemma}\label{le:subseq-diff-upbd}
Suppose Assumptions~\ref{asp:olse} and \ref{asp:rdm-explr} hold, and $\hat{w}^{(k)}_i$ and $\hat{w}^{(k+1)}_i$ are the solutions to the OLSE at the $k$-th and $(k+1)$-th iteration, respectively. 
If the matrix $\frac{1}{k}\sum_{t=0}^{k-1}\ell^{(t)}_i(\ell^{(t)}_i)^T$ is positive definite and all its eigenvalues are lower bounded by $\ubar{D}_i$, then $\norm{\hat{w}^{(k+1)}_i - \hat{w}^{(k)}_i}_2 \leq C_{w,i}(\ubar{D}_ik)^{-1}$, where $C_{w,i}$ is some constant determined by the diameters of $\mathcal{X}^+_i$, $\mathcal{W}_i$, and the ranges of $\xi_i$ and $\delta_i$. 
\end{lemma}
\begin{proof}
Note that $\hat{w}^{(k)}$ is defined as:
\begin{align}\label{eq:olse-linear-k}
\hat{w}^{(k)}_i \coloneqq \argmin_{w_i \in \mathcal{W}_i} \frac{1}{k}\sum_{t=0}^{k-1}
(s^{(t)}_{i} - \langle \ell^{(t)}_{i}, w_i\rangle)^2. 
\end{align}
The associated variational inequality $\text{VI}(\mathcal{W}_i, F_i^{(k)})$ of \eqref{eq:olse-linear-k} has the continuous operator $F_i^{(k)}:\mathcal{W}_i \to \rset{n^+_i+1}{}$ given by
\begin{align}\label{eq:pf-seqdiff-fk}
F^{(k)}_i: w \mapsto \epsilon \cdot \Big( \frac{1}{k}\sum_{t=0}^{k-1}\ell^{(t)}_i(\ell^{(t)}_i)^Tw - \frac{1}{k}\sum_{t=0}^{k-1}s^{(t)}_{i}\ell^{(t)}_i \Big),
\end{align}
where $\epsilon > 0$ is an arbitrary positive constant. 
Then, $\hat{w}^{(k)}_i$ belongs to the solution set of $\text{VI}(\mathcal{W}_i, F_i^{(k)})$, i.e., $\hat{w}^{(k)}_i \in \text{SOL}(\mathcal{W}_i, F_i^{(k)})$. 
Likewise, for the solution $\hat{w}^{(k+1)}_i$ of the problem \eqref{eq:olse-linear-copy}, $\hat{w}^{(k+1)}_i$ is a solution to the $\text{VI}(\mathcal{W}_i, F_i^{(k)})$, i.e., $\hat{w}^{(k+1)}_i \in \text{SOL}(\mathcal{W}_i, F_i^{(k+1)})$, where the operator $F_i^{(k+1)}$ is given as follows:
\begin{align}\label{eq:pf-seqdiff-fk+1}
\begin{split}
F^{(k+1)}_i: w \mapsto & \epsilon \cdot \frac{k+1}{k}\Big( \frac{1}{k+1}\sum_{t=0}^{k-1}\ell^{(t)}_i(\ell^{(t)}_i)^Tw - \frac{1}{k+1}\sum_{t=0}^{k-1}s^{(t)}_{i}\ell^{(t)}_i \\
& + \frac{1}{k+1}\ell^{(k)}_i(\ell^{(k)}_i)^Tw - \frac{1}{k+1}s^{(k)}_{i}\ell^{(k)}_i \Big) \\
&= \epsilon \cdot \Big( \frac{1}{k}\sum_{t=0}^{k-1}\ell^{(t)}_i(\ell^{(t)}_i)^Tw - \frac{1}{k}\sum_{t=0}^{k-1}s^{(t)}_{i}\ell^{(t)}_i + \frac{1}{k}\ell^{(k)}_i(\ell^{(k)}_i)^Tw - \frac{1}{k}s^{(k)}_{i}\ell^{(k)}_i \Big),
\end{split}
\end{align}
with the arbitrary constant $\epsilon > 0$. 
By \citeApndx[Prop.~1.5.8]{facchinei2007finite} and the compactness of $\mathcal{W}_i$, we have:
\begin{align*}
\hat{w}^{(k)}_i \in \text{SOL}(\mathcal{W}_i, F_i^{(k)}) &\Leftrightarrow \hat{w}^{(k)}_i - \proj_{\mathcal{W}_i}(\hat{w}^{(k)}_i - F_i^{(k)}(\hat{w}^{(k)}_i)) = 0, \\
\hat{w}^{(k+1)}_i \in \text{SOL}(\mathcal{W}_i, F_i^{(k+1)}) &\Leftrightarrow \hat{w}^{(k+1)}_i - \proj_{\mathcal{W}_i}(\hat{w}^{(k+1)}_i - F_i^{(k+1)}(\hat{w}^{(k+1)}_i)) = 0. 
\end{align*}
With the above equivalent relations in hand, we can establish an upper bound for the difference $\norm{\hat{w}^{(k+1)}_i - \hat{w}^{(k)}_i}_2$ as follows:
\begin{align*}
\begin{split}
\norm{\hat{w}^{(k+1)}_i - \hat{w}^{(k)}_i}_2 &= \norm{\proj_{\mathcal{W}_i}(\hat{w}^{(k)}_i - F_i^{(k)}(\hat{w}^{(k)}_i)) - \proj_{\mathcal{W}_i}(\hat{w}^{(k+1)}_i - F_i^{(k+1)}(\hat{w}^{(k+1)}_i))}_2 \\
& \overset{(a)}{\leq} \norm{(\hat{w}^{(k)}_i - F_i^{(k)}(\hat{w}^{(k)}_i)) - (\hat{w}^{(k+1)}_i - F_i^{(k+1)}(\hat{w}^{(k+1)}_i))}_2 \\
& \overset{(b)}{=} \norm{ \hat{w}^{(k)}_i - \hat{w}^{(k+1)}_i - \epsilon \cdot \Big( \frac{1}{k}\sum_{t=0}^{k-1}\ell^{(t)}_i(\ell^{(t)}_i)^T\Big)(\hat{w}^{(k)}_i - \hat{w}^{(k+1)}_i)  \\
& \qquad + \epsilon\cdot(\frac{1}{k}\ell^{(k)}_i(\ell^{(k)}_i)^T\hat{w}^{(k+1)}_i - \frac{1}{k}s^{(k)}_{i}\ell^{(k)}_i)}_2 \\
& \leq \underbrace{\norm{I - \epsilon \cdot \Big( \frac{1}{k}\sum_{t=0}^{k-1}\ell^{(t)}_i(\ell^{(t)}_i)^T\Big)}_2 \cdot \norm{\hat{w}^{(k+1)}_i - \hat{w}^{(k)}_i}_2}_{(\RN{1})} \\
& \qquad + \underbrace{\frac{\epsilon}{k}\norm{\ell^{(k)}_i(\ell^{(k)}_i)^T\hat{w}^{(k+1)}_i-s^{(k)}_{i}\ell^{(k)}_i}_2}_{(\RN{2})},
\end{split}
\end{align*}
where $(a)$ follows from the nonexpansiveness of the projection onto a convex set; $(b)$ is a result of using \eqref{eq:pf-seqdiff-fk} and \eqref{eq:pf-seqdiff-fk+1}. 
Next, we can establish a bound for $(\RN{2})$ by noticing that $\mathcal{X}^+_i$, $\mathcal{W}_i$, and the ranges of $\xi_i$ and $\delta_i$ are all bounded, and we can find a constant $C_{w,i}$, such that $\norm{\ell^{(k)}_i(\ell^{(k)}_i)^T\hat{w}^{(k+1)}_i-s^{(k)}_{i}\ell^{(k)}_i}_2 \leq C_{w,i}$. 
To derive an appropriate upper bound for $(\RN{1})$, we choose $\epsilon$ sufficiently small, such that for all possible $\ell^{(t)}_i$ in the feasible set, the positive eigenvalue of $\epsilon \cdot \ell^{(t)}_i(\ell^{(t)}_i)^T$ is less than $1$, i.e., $\epsilon \norm{\ell^{(t)}_i}^2_2 < 1$. 
The existence of $\epsilon$ is guaranteed by the boundedness of $\mathcal{X}^+_i$. 
Combining this with the assumption that all eigenvalues of $\frac{1}{k}\sum_{t=0}^{k-1}\ell^{(t)}_i(\ell^{(t)}_i)^T$ are lower bounded by $\ubar{D}_i$ yields:
\begin{align*}
0 \prec I - \epsilon \cdot \Big( \frac{1}{k}\sum_{t=0}^{k-1}\ell^{(t)}_i(\ell^{(t)}_i)^T \Big) \prec (1 - \ubar{D}_i\epsilon)I. 
\end{align*}
The upper bounds for $(\RN{1})$ and $(\RN{2})$ together lead to the following result for the difference
\begin{align*}
\begin{split}
\norm{\hat{w}^{(k+1)}_i - \hat{w}^{(k)}_i}_2 \leq (1 - \ubar{D}_i\epsilon)\norm{\hat{w}^{(k+1)}_i - \hat{w}^{(k)}_i}_2 + \frac{\epsilon C_{w,i}}{k}, 
\end{split}
\end{align*}
and hence
\begin{align*}
\begin{split}
\norm{\hat{w}^{(k+1)}_i - \hat{w}^{(k)}_i}_2 \leq \frac{C_{w,i}}{\ubar{D}_i}\cdot \frac{1}{k}, 
\end{split}
\end{align*}
which completes the proof of our claim. 
\end{proof}

\begin{proof}[Proof of Theorem~\ref{thm:obsv-num-est-err}]
By Lemma~\ref{le:lower-bd} and Assumption~\ref{asp:feaset}, there exists a pair of constants $(\ubar{D}_i, \bar{D}_i)$ such that \eqref{eq:lld-cond} is fulfilled a.s. 
Denote the $\sigma$-field $\mathcal{F}^{\delta}_{\infty}$ generated by the random initialization and random exploration factors as follows:
\begin{align*}
\mathcal{F}^{\delta}_{\infty} \coloneqq \sigma\{y^{(0)}, \hat{w}^{(0)}, \{\delta^{(0)}_i\}, \ldots, \{\delta^{(k)}_i\}, \ldots \},
\end{align*}
which distinguishes from the sub-$\sigma$-field $\mathcal{F}_{k}$ defined in the main text, i.e., 
\begin{align*}
\mathcal{F}_{k} \coloneqq \sigma\{y^{(0)}, \hat{w}^{(0)}, \{{\delta}^{(0)}_i\}, \{{\xi}^{(0)}_i\}, \ldots, \{{\delta}^{(k-2)}_i\}, \{{\xi}^{(k-2)}_i\}\}. 
\end{align*}
Let $E^{(k)}_i$ denote the event $\sqrt{k}\norm{\Delta w^{(k)}_i}_2 \geq k^{\alpha_2}$, where $0 < \alpha_2 < \frac{1}{2}$. 
Let $\hat{\Omega}^{\delta} \in \mathcal{F}^{\delta}_{\infty}$ denote a sample subset with probability one, and for any $\hat{\omega}^{\delta} \in \hat{\Omega}^{\delta}$, there exists a finite index $\tilde{K}(\hat{\omega}^{\delta})$ such that for all $k \geq \tilde{K}(\hat{\omega}^{\delta})$, $\mathbb{P}_{w^*_i}\{E^{(k)}_i \mid \hat{\omega}^{\delta}\} \leq D'_i \exp(-D^*_i k^{2\alpha_2})$, based on the conclusion of Theorem~\ref{thm:lld}. 
We then investigate the summability of this sequence of probability measures: 
\begin{align*}
\sum_{k=\tilde{K}}^{+\infty}\mathbb{P}_{w^*_i}\{E^{(k)}_i \mid \hat{\omega}^{\delta}\} \leq \sum_{k=\tilde{K}}^{+\infty} D'_i \exp(-D^*_i k^{2\alpha_2})
\end{align*}
Note that
\begin{align*}
\lim_{k \to \infty}\frac{\exp(-D_i^*k^{2\alpha_2}))}{1/k^2} 
= \lim_{k \to \infty}\frac{k^2}{\exp(D_{i}^* k^{2\alpha_2})} 
= \lim_{\mu \to \infty}\frac{\mu^{1/\alpha_2}}{a^{\mu}} = 0,
\end{align*}
where we let $\mu \coloneqq k^{2\alpha_2}$ and $a \coloneqq e^{D^*_i} > 1$. 
The above relation further implies 
\begin{align*}
\sum_{k=\tilde{K}}^{\infty} \mathbb{P}_{w^*_i}\{E_i^{(k)} \mid \hat{\omega}^{\delta}\}  < C_1\sum_{k=\tilde{K}}^{\infty} 1/k^2 < +\infty. 
\end{align*}
where $C_1$ denotes some constant. 
The Borel-Cantelli lemma states that: if the sum of the probabilities of the event ${E^{(k)}_{i}}$ conditioning on $\hat{\omega}^{\delta}$ is finite, i.e., $\sum_{t=0}^{\infty}\mathbb{P}_{w^*_i}\{E_i^{(k)} \mid \hat{\omega}^{\delta}\} < \infty$, then the probability that infinitely many of them occur is 0, i.e., 
\begin{align*}
P\{\limsup_{t \to \infty} E_i^{(k)} \mid \hat{\omega}^{\delta}\} = P\{\cap_{t=0}^{\infty}\cup_{k=t}^{\infty}E_i^{(k)} \mid \hat{\omega}^{\delta}\} = 0.
\end{align*}
Hence, almost surely, $\{E_i^{(k)}\}_{k \in \nset{}{}}$ conditioning on $\hat{\omega}^{\delta}$ is true for only a finite number of $k \in \nset{}{}$.
In what follows, we focus on the sample set $\hat{\Omega}$ with probability one, and for each $\hat{\omega} \in \hat{\Omega}$, (i) there exists an index $K_{\text{id}}(\hat{\omega})$ such that for all $k \geq K_{\text{id}}(\hat{\omega})$, the lower bound $\ubar{D}_i$ established in Lemma~\ref{le:lower-bd} holds, i.e., $\frac{1}{k}\sum_{t=0}^{k-1}\ell^{(t)}_i(\ell^{(t)}_i)^T \succ \ubar{D}_i I$, and (ii) $\{E_i^{(k)}\}_{k \in \nset{}{}}$ conditioning on $\hat{\omega}$ is true for only a finite number of $k \in \nset{}{}$, denoted by the set $S_E(\hat{\omega})$ with cardinality $\abs{S_E(\hat{\omega})} < \infty$. 
After picking $K_{\text{lse}}(\hat{\omega}) = \max\{k: k \in S_E(\hat{\omega}) \cup \{K_{id}(\hat{\omega})\}\}$, for all $k \geq K_{\text{lse}}(\hat{\omega})$, we have
\begin{align*}
\begin{split}
& \expt{}{\norm{\Delta \hat{w}^{(k+1)}_i}_2 \mid \mathcal{F}_{k+1}}(\hat{\omega}) \leq \expt{}{\norm{\hat{w}^{(k+1)}_i - \hat{w}^{(k)}_i}_2 + \norm{\hat{w}^{(k)}_i - w^*_i}_2 \mid \mathcal{F}_{k+1} }(\hat{\omega}) \\
& = \expt{}{\norm{\hat{w}^{(k+1)}_i - \hat{w}^{(k)}_i}_2 \mid \mathcal{F}_{k+1} }(\hat{\omega}) + \expt{}{ \norm{\hat{w}^{(k)}_i - w^*_i}_2 \mid \mathcal{F}_{k+1} }(\hat{\omega}) \\
& \overset{(a)}{\leq} \frac{C_{w,i}}{\ubar{D}_i}\cdot\frac{1}{k} + \norm{\hat{w}^{(k)}_i - w^*_i}_2(\hat{\omega}) \\
& \overset{(b)}{\leq} \frac{C_{w,i}}{\ubar{D}_i}\cdot\frac{1}{k} + k^{\alpha_2 - 0.5}, 
\end{split}
\end{align*}
where $(a)$ follows from Lemma~\ref{le:subseq-diff-upbd} and the fact that $\norm{\hat{w}^{(k)}_i - w^*_i}_2$ is $\mathcal{F}_{k+1}$-measurable; 
$(b)$ follows from the result that $E^{(k)}_i$ conditioning on $\hat{\omega}$ is false for all $k \geq K_{\text{lse}}(\hat{\omega})$. 
We can further relax the above upper bound to be $\expt{}{\norm{\Delta \hat{w}^{(k+1)}_i}_2 \mid \mathcal{F}_{k+1}}(\hat{\omega}) \leq C_{\Delta, i}(k+1)^{\alpha_2 - 0.5}$, where $C_{\Delta, i}$ can be chosen as $C_{\Delta, i} \coloneqq \frac{C_{w,i}}{\ubar{D}_i} + 2$ independent of $\hat{\omega}$ and $k$. 

\end{proof}

\subsection{Extension to the Learning of Generalized Nash Equilibria}

In this appendix, we further consider an extension to the problem in the main text, where the local decision $x_i$ of player $i$ should be subject to three types of constraints: the local constraints $\mathcal{X}_i \subseteq \rset{n_i}{}$, the locally coupled constraints involving the decisions of its in-neighbors $\mathcal{X}^+_i \coloneqq \{[x_i; x^+_i] \mid h_i(x_i, x^+_i) \leq \bzero_{m_i} \} \subseteq \rset{n^+_i}{}$, and the global resource constraints $\sum_{j \in \playerN}A_jx_j \leq c$, where $A_j \in \rset{m \times n_j}{}$ with $m$ denoting the number of global affine coupling constraints, and $c \in \rset{m}{}$ is a constant vector which usually represents the quantities of available resources. 
The feasible set for the stack decision vectors is defined as $\hat{\mathcal{X}} \coloneqq \{[x_i]_{i \in \playerN} \in \rset{n}{}\mid \sum_{i \in \playerN} A_ix_i \leq c, x_i \in \mathcal{X}_i, [x_i; x^+_i] \in \mathcal{X}^+_i, \forall i \in \playerN\}$, 
and the feasible set of player $i$ given the decisions of others is given by $\hat{\mathcal{X}}_i(\{x_j\}_{j \in \playerN \backslash \{i\}}) \coloneqq \{x_i \in \rset{n_i}{} \mid  x_i \in \mathcal{X}_i, [x_i; x^{+}_i] \in \mathcal{X}^+_i, A_ix_i \leq c - \sum_{j \in \neighbN{}{}\backslash\{i\}}A_jx_j\}$. 

Altogether, the local stochastic optimization problem of player $i$ can be formally written as:
\begin{equation}\label{eq:pf-optprob-g}
\begin{cases}
\minimize_{x_i \in \mathcal{X}_i} & \mathbb{J}_i(x_i; x^+_i, w^*_i) \\
\subj & h_i(x_i, x^+_i) \leq \bzero_{m_i}, \; A_ix_i \leq c - \sum_{j \in \playerN \backslash \{i\}} A_jx_j. 
\end{cases}
\end{equation}

The solution concept of the problem \eqref{eq:pf-optprob-g} we focus on in this section is stochastic generalized Nash equilibria (SGNEs) (\citeApndx{franci2020distributed, ravat2011characterization}), whose definition is given as follows.
\begin{definition}
The collective decision $x^* \in \hat{\mathcal{X}}$ is called a stochastic generalized Nash equilibrium (SGNE) if no player can benefit by unilaterally deviating from $x^*$, i.e., $\forall i \in \playerN$, $\mathbb{J}_i(x^*_i; x^{*+}_{i}, w^*_i) \leq \mathbb{J}_i(x_i; x^{*+}_{i}, w^*_i)$ for any deviation $x_i \in \hat{\mathcal{X}}_i(\{x^*_j\}_{j \in \playerN \backslash \{i\}})$. 
\end{definition}

We then make the following regularity assumptions concerning the objective functions and feasible sets.
Under these assumptions, SGNEs seeking problems can be recast as the corresponding generalized quasi-variational inequality (GQVI) \citeApndx[Ch.~16.2]{palomar2010convex}. 
\begin{assumption}\label{asp:obj-g} (Local Objectives)
For each $i \in \playerN$, given any fixed sample $\omega_i \in \Omega$ and the precise parameters $w^*_i$, the scenario-based and expected-value objectives $J_i$ and $\mathbb{J}_i$ satisfy:

(i) $J_i(x_i; s_i(x^{+}_{i}, \xi_i; w^*_i))$ is convex in $x_i$ given any fixed $x^{+}_{i}$;

(ii) $J_i(x_i; s_i(x^{+}_{i}, \xi_i; w^*_i))$ is proper and continuous in $x_i$ and $x^{+}_{i}$;

(iii) $\mathbb{J}_i$ can be written as $\mathbb{J}_i(x_i; x^{+}_{i}, \hat{w}_i) = \mathbbm{f}_i(x_i; x^{+}_{i}) + \mathbbm{g}_i(x_i; x^{+}_{i}, \hat{w}_i)$, where for any fixed $x^{+}_{i}$, $\mathbbm{f}_i$ is convex in $x_i$ and $\mathbbm{g}_i$ is differentiable in $x_i$. Moreover, $\mathbbm{g}_i$ is continuous in $\hat{w}_i$ and $\nabla_{x_i}\mathbbm{g}_i$ is Lipschitz in $\hat{w}_i$ with the constant $\alpha_{g,i}\norm{x^{+}_{i}}_2 + \beta_{g,i}$ ($\alpha_{g,i}, \beta_{g,i} \geq 0$) on $\mathcal{W}_i$ for any fixed $x_i \in \mathcal{X}_i$ and $x^{+}_{i}$ . 

\end{assumption}

\begin{assumption}\label{asp:feaset-g} (Feasible Sets)
For each $i \in \playerN$, $\mathcal{X}_i$ is nonempty, compact, and convex. 
The function $h_i$ is an extended-real-valued, closed convex proper function. 
The collective feasible set $\hat{\mathcal{X}}$ is nonempty, and satisfies the the Mangasarian-Fromovitz constraint qualification (MFCQ) \citeApndx[Sect.~16.2.1]{palomar2010convex}.
\end{assumption}

In this paper, we further restrict ourselves to the variational stochastic generalized Nash equilibria (v-SGNEs) \citeApndx{facchinei2010generalized, facchinei2007generalized}, a subclass of SGNEs. 
The motivation of this restriction is that v-SGNEs possess better economic fairness compared with other SGNEs. 
Moreover, there exist a variety of tools that have been developed for solving the corresponding generalized variational inequalities (GVI), the solution of which is the v-SGNEs of the original problem \eqref{eq:pf-optprob-g}. 
For a more detailed discussion about SGNE and v-SGNE, we refer the interested reader to \citeApndx{huang2021sgnep}.
As has been shown in \citeApndx[Prop.~12.4, Sect.~12.2.3]{palomar2010convex}, to compute a v-SGNE of \eqref{eq:pf-optprob-g}, we can instead solve the Karush-Kuhn-Tucker (KKT) problem of its corresponding (GVI), which is given as follows:
\begin{align}\label{eq:kkt-gvi-g}
\begin{split}
& \bzero \in \partial_{x_i}\mathbb{J}_i(x_i; x^{+}_{i}, w^*_i) + \sum_{j \in \neighbN{-}{i} \cup \{i\}} \big(\partial_{x_i}h_j(x_j, x^{+}_j)\big)^T\rho_j + A_i^T\lambda + N_{\mathcal{X}_i}(x_i), \forall i \in \playerN, \\
& \bzero \in -h_i(x_i, x^{+}_i) + N_{\rset{m_i}{+}}(\rho_i), \forall i \in \playerN, \\
& \bzero \in -(Ax - c) + N_{\rset{m}{+}}(\lambda), 
\end{split}
\end{align}
where $\lambda$ is the Lagrange multiplier for the global resource constraints, and $\rho_i$ denotes the Lagrange multiplier for each locally coupled constraints $h_i(x_i; x^+_i) \leq \bzero$. 
It is worth mentioning that the solution of \eqref{eq:kkt-gvi-g} is a subset of that of \eqref{eq:pf-optprob-g}, and under some circumstances, the game in \eqref{eq:pf-optprob-g} may have a SGNE, yet \eqref{eq:kkt-gvi-g} admits an empty solution set. 
Because of this, we make the following assumption concerning the existence of a v-SGNE. 

\begin{assumption}\label{asp:exist-vsgne}
The SGNEP considered admits a nonempty set of v-SGNEs. 
\end{assumption}

We can conduct similar distributed reformulation as in Section~\ref{subsect:dist-sol-ppa} and introduce an incidence matrix $\tilde{B}$ for the dependency graph. 
Let $B$ denote the incidence matrix of the original communication graph $\mathcal{G}$ and $B_m \coloneqq B \otimes I_m$. 
For the problem \eqref{eq:pf-optprob-g}, its associated zero-finding operator $\optT$ is given as follows:

\begin{align}
\begin{split}
\optT: 
\psi \mapsto & \partial ( \sum_{i \in \playerN} \mathbb{J}_i(y^i_i; {y}^+_i, w^*_i) + \iota_{\mathcal{X}_i}(y^i_i) + \iota_{\mathcal{X}^+_i}(y_i) + \iota_{\rset{m}{+}}(\lambda_i)) \\
& + \begin{bmatrix}
\rho\tilde{L} & (\Lambda\mathcal{R})^T & \tilde{B} & \bzero \\
- \Lambda\mathcal{R} & \bzero & \bzero & B_m \\
-\tilde{B}^T & \bzero & \bzero & \bzero \\ 
\bzero & -B_m^T & \bzero & \bzero
\end{bmatrix}
\begin{bmatrix}
y \\ \blambda \\ \mu \\ z
\end{bmatrix}
+ \begin{bmatrix} \bzero \\ \boldsymbol{c} \\ \bzero \\ \bzero \end{bmatrix}
,
\end{split}
\end{align}
where the stack vector $\psi$ is defined as $\psi \coloneqq [y; \blambda; \mu; z]$; $\blambda$ denotes the stack of local estimates of the global resource constraints multiplier $\lambda \in \rset{m}{+}$, i.e., $\blambda \coloneqq [\lambda_1; \cdots; \lambda_N]$; $\mu$ is the stack of the consensus multipliers $\mu_{ji}$, which enforce $y^j_i = y^j_j$ for all $i \in \playerN$ and $j \in \neighbN{+}{i}$; 
likewise, $z \coloneqq [[z_{ji}]_{j \in \neighbN{+}{i}}]_{i \in \playerN}$, where each $z_{ji}$ serves as the multiplier for the consensus constraint $\lambda_j = \lambda_i$. 
The vector $\boldsymbol c \coloneqq [c_1; \cdots; c_N] \in \rset{mN}{+}$ satisfies $\sum_{i \in \playerN} c_i = c$. 
We let each player maintain $\psi_i \coloneqq [y_i; \lambda_i; [\mu_{ji}]_{j \in \neighbN{+}{i}}, [z_{ji}]_{j \in \neighbN{+}{i}}]$. 
For more details about the construction of $\optT$, we refer the interested readers to \citeApndx[Sec.~II]{huang2021distributed} 

\begin{theorem}
Suppose Assumptions~\ref{asp:commtopo}, \ref{asp:obj-g} and \ref{asp:feaset-g} hold, and there exists $\psi^* \coloneqq [y^*; \blambda^*; \mu^*, z^*] \in \zer{\optT}$. 
Then $\tilde{L}y^* = \bzero$, $\blambda^* = \bone_N \otimes \lambda^*$, and there exists a set of Lagrange multipliers $\{\rho^*_i\}_{i \in \playerN}$ such that the tuple $(\{y^{i*}_i\}_{i \in \playerN}, \lambda^*, \{\rho^*_i\}_{i \in \playerN})$ satisfies the KKT conditions \eqref{eq:kkt-gvi-g} for v-SGNE with each $x^*_i$ replaced by $y^{i*}_i$. 
Conversely, if the KKT problem in \eqref{eq:kkt-gvi-g} admits a solution $(\{y^{i\dagger}_i\}_{i \in \playerN}, \lambda^\dagger, \{{\rho}^\dagger_i\}_{i \in \playerN})$, then there exist $\mu^\dagger$ and $z^\dagger$ such that by constructing $y^\dagger$ with the local estimate $y^{j\dagger}_i = y^{j\dagger}_j$ for all $i \in \playerN$ and $j \in \neighbN{+}{i}$, the stack vector $\psi^\dagger \coloneqq [y^\dagger; \bone_N \otimes \lambda^\dagger; \mu^\dagger, z^\dagger] \in \zer{\optT}$. 
\end{theorem}

\begin{proof}
Suppose there exists $\psi^* \in \zer{\optT}$. 
We then have $\tilde{B}^Ty = \bzero$ and $B_m^T \blambda^* = \bzero$, which implies $y^{j*}_i = y^{j*}_j$ and $\blambda^* = \bone_N \otimes \lambda^*$, and thus $\rho\tilde{L}y^* = \bzero$. 
We next focus on the separable inclusions involving $y^*$ and $\blambda^*$. 
Note that $\mathcal{X}^+_i \coloneqq \{[x_i; x^+_i] \mid h_i(x_i, x^+_i) \leq \bzero_{m_i} \}$ and in the operator $\optT$, we apply the normal cone of $\mathcal{X}^+_i$ to $y_i$ instead of the decisions $y^i_i$ and $\{y^j_j\}_{j \in \neighbN{+}{i}}$. 
Moreover, an explicit formulation of the normal cone is given by
\begin{align*}
N_{\mathcal{X}^+_i}([x_i; x^+_i]) = \begin{cases}
\varnothing, & \text{if}\;h_i(x_i, x^+_i) > \bzero, \\
\bzero, & \text{if}\;h_i(x_i, x^+_i) < \bzero, \\
g^T\rho_i, \forall \rho_i \in \rset{m_i}{+}, g \in [\partial_{x_j} h_i(x_i, x^+_i)]_{j \in \{i\}\cup \neighbN{+}{i}} & \text{if}\;h_i(x_i, x^+_i) \leq \bzero,
\end{cases}
\end{align*}
where $g$ is a matrix with dimension $m_i \times (n_i + n^+_i)$, and $\rho_i$ has each of its entry equal to zero if the corresponding inequality constraint is inactive and otherwise positive. 
For each player $i \in \playerN$, the inclusion involving the local decision $y^{i*}_i$ can be expressed as:
\begin{align}\label{eq:g-eq-pf-local-dec}
\begin{split}
\bzero \in& \partial_{x_i}\mathbb{J}_i(y^{i*}_i; y^{+*}_i, w^*_i) + N_{\mathcal{X}_i}(y^{i*}_i) + (\partial_{x_i}h_i(y^{i*}_i, y^{+*}_i))^T\rho^*_i + A^T_i\lambda^* + (-\sum_{j \in \neighbN{-}{i}}\mu^{*}_{ij}), \\
\bzero \in& -h_i(y^{i*}_i, y^{+*}_i) + N_{\rset{m_i}{+}}(\rho^*_i), 
\end{split}
\end{align}
where we let $y^{+*}_i \coloneqq [y^{j*}_i]_{j \in \neighbN{+}{i}}$, and $\rho^*_i \in \rset{m_i}{+}$ is some implicit Lagrange multiplier induced by the normal cone $N_{\mathcal{X}^+_i}$. 
For each out-neighbor $j$ of player $i$ ($j \in \neighbN{-}{i}$), the inclusion involving the local estimate $y^{i*}_j$ can be expressed as:
\begin{align}\label{eq:g-eq-pf-local-est}
    \bzero \in (\partial_{x_i}h_j(y^{j*}_j, y^{+*}_j))^T\rho^*_j + \mu^*_{ij}, 
\end{align}
where $\rho^*_j \in \rset{m_i}{+}$ is defined similarly. 
Substituting $\mu^*_{ij}$ in \eqref{eq:g-eq-pf-local-dec} with \eqref{eq:g-eq-pf-local-est}, we obtain
\begin{align}
\begin{split}
\bzero \in& \partial_{x_i}\mathbb{J}_i(y^{i*}_i; y^{+*}_i, w^*_i) + N_{\mathcal{X}_i}(y^{i*}_i) + \sum_{j \in \{i\}\cup\neighbN{-}{i}}(\partial_{x_i}h_j(y^{j*}_j, y^{+*}_j))^T\rho^*_j + A^T_i\lambda^* \\
\bzero \in& -h_i(y^{i*}_i, y^{+*}_i) + N_{\rset{m_i}{+}}(\rho^*_i).
\end{split}
\end{align}
Furthermore, the rows associated with $\blambda^*$ can be expressed as:
\begin{align*}
\begin{split}
& \bzero_{mN} \in N_{\rset{mN}{+}}(\blambda^*) - \Lambda\mathcal{R}y^* + \boldsymbol{c} + B_mz^* \\
& \Rightarrow \bzero_{m} \in (\bone^T_m \otimes I_m)(N_{\rset{mN}{+}}(\blambda^*) - \Lambda\mathcal{R}y^* + \boldsymbol{c} + B_mz^*)\\
& \Leftrightarrow \bzero_{m} \in N_{\rset{m}{+}}(\lambda^*) - \sum_{i \in \playerN} A_iy^{i*}_i + c. 
\end{split}
\end{align*}
Altogether, $(\{y^{i*}_i\}_{i \in \playerN}, \lambda^*, \{\rho^*_i\}_{i \in \playerN})$ satisfies the KKT conditions of a v-SGNE \eqref{eq:kkt-gvi-g}. 

In the other direction, assume the KKT systems \eqref{eq:kkt-gvi-g} admits a solution $(\{y^{i\dagger}_i\}_{i \in \playerN}, \lambda^\dagger, \{{\rho}^\dagger_i\}_{i \in \playerN})$. 
We construct the vector $y^\dagger$ by letting each local estimate $y^{j\dagger}_i = y^{j\dagger}_j$, for all $i \in \playerN$ and $j \in \neighbN{+}{i}$ and let $\blambda^\dagger = \bone_N \otimes \lambda^\dagger$. 
For each $i \in \playerN$ and $j \in \neighbN{-}{i}$, we let $\mu^\dagger_{ij}$ be such that
\begin{align*}
\begin{split}
& \mu^\dagger_{ij} \in -(\partial_{x_i} h_j(y^{j\dagger}_j, y^{+\dagger}_j))^T\rho^\dagger_j, \;\text{and} \\
& \bzero \in \partial_{x_i}\mathbb{J}_i(y^{i\dagger}_i; y^{+\dagger}_i, w^*_i) + N_{\mathcal{X}_i}(y^{i\dagger}_i) + (\partial_{x_i}h_i(y^{i\dagger}_i, y^{+\dagger}_i))^T\rho^\dagger_i + A^T_i\lambda^\dagger + (-\sum_{j \in \neighbN{-}{i}}\mu^{\dagger}_{ij}), \\
\end{split}
\end{align*}
with $\rho^\dagger_i \in \rset{m_i}{+}$ and satisfying $\bzero \in -h_i(y^{i\dagger}_i, y^{+\dagger}_i) + N_{\rset{m_i}{+}}(\rho^\dagger_i)$. 
We can further replace the term $(\partial_{x_i}h_i(y^{i\dagger}_i, y^{+\dagger}_i))^T\rho^\dagger_i$ and similar terms in the inclusions of local estimates with a normal cone operator $N_{\mathcal{X}^+_i}$. 
What remains to be shown is that there exists a vector $z^\dagger$ such that 
\begin{align*}
    \bzero_{mN} \in N_{\rset{mN}{+}}(\blambda^\dagger) - \Lambda\mathcal{R}y^\dagger + \boldsymbol{c} + B_mz^\dagger.
\end{align*}
The KKT conditions \eqref{eq:kkt-gvi-g} suggests $\bzero \in -(\sum_{i \in \playerN} A_iy^{i\dagger}_i - c) + N_{\rset{m}{+}}(\lambda^\dagger)$. 
Let $r \coloneqq c - \sum_{i \in \playerN} A_iy^{i\dagger}_i \geq \bzero_{m}$ and $\bzero_{m} \leq \lambda^\dagger \perp r \geq \bzero_{m}$. 
Notice that $B_m z$ spans the null space of $(\bone_N^T \otimes I_m)$. 
Hence, we can find a $z^\dagger$, such that $[c_i - A_iy^\dagger_i - \frac{1}{N}r]_{i \in \playerN} + B_mz^\dagger = \bzero_{Nm}$.
It immediately follows that 
\begin{align*}
\bzero_{mN} \leq \blambda^\dagger \perp \boldsymbol{c} - \diagA\mathcal{R}\by^\dagger + B_m z^\dagger \geq \bzero_{mN},
\end{align*}
which completes the proof that there exist $\mu^\dagger$ and $z^\dagger$ such that the stack vector $\psi^\dagger \coloneqq [y^\dagger; \bone_N \otimes \lambda^\dagger; \mu^\dagger, z^\dagger] \in \zer{\optT}$
\end{proof}

Note that with locally coupling constraints, the resolvent of $\optT$ can not be expressed explicitly as we did in Section~\ref{subsect:dist-sol-ppa} and carried out distributedly. 
To handle the computational difficulty induced by this, we resort to the Douglas-Rachford splitting \citeApndx[Sec.~26.3]{BauschkeHeinzH2017CAaM}, which separates the target operator into two monotone operators.  
The convergence of the Douglas-Rachford iteration relies on the nonexpansiveness of the reflect resolvents of the split operators. 
The operator $\optT$ considered can be split into the following two operators.
\begin{align}
\begin{split}
\optA: \begin{bmatrix} y \\ \blambda \\ \mu \\ z \end{bmatrix} \mapsto 
\begin{bmatrix} \extgjacob_{w^*}(y) + N_{\tilde{\mathcal{X}}}(y) \\ \boldsymbol c \\ \bzero \\ \bzero \end{bmatrix} + \frac{1}{2}
\begin{bmatrix}
\rho\tilde{L} & (\Lambda\mathcal{R})^T & \tilde{B} & \bzero \\
- \Lambda\mathcal{R} & \bzero & \bzero & B_m \\
-\tilde{B}^T & \bzero & \bzero & \bzero \\ 
\bzero & -B_m^T & \bzero & \bzero
\end{bmatrix}\begin{bmatrix} y \\ \blambda \\ \mu \\ z \end{bmatrix}, 
\end{split}
\end{align}
\begin{align}
\begin{split}
\optB: \begin{bmatrix} y \\ \blambda \\ \mu \\ z \end{bmatrix} \mapsto 
\begin{bmatrix} [N_{\mathcal{X}^+_i}(y_i)]_{i \in \playerN} \\ N_{\rset{mN}{+}}(\blambda) \\ \bzero \\ \bzero \end{bmatrix} + \frac{1}{2}
\begin{bmatrix}
\rho\tilde{L} & (\Lambda\mathcal{R})^T & \tilde{B} & \bzero \\
- \Lambda\mathcal{R} & \bzero & \bzero & B_m \\
-\tilde{B}^T & \bzero & \bzero & \bzero \\ 
\bzero & -B_m^T & \bzero & \bzero
\end{bmatrix}\begin{bmatrix} y \\ \blambda \\ \mu \\ z \end{bmatrix}, 
\end{split}
\end{align}
The associated design matrix to distribute the computation is given by:
\begin{align}
\begin{split}
\Phi \coloneqq 
\begin{bmatrix}
\btau^{-1}_1 - \frac{1}{2}\rho\tilde{L} & -\frac{1}{2}(\Lambda\mathcal{R})^T & -\frac{1}{2}\tilde{B} & \bzero \\
- \frac{1}{2}\Lambda\mathcal{R} & \btau^{-1}_2 & \bzero & -\frac{1}{2}B_m \\
-\frac{1}{2}\tilde{B}^T & \bzero & \btau^{-1}_3 & \bzero \\ 
\bzero & -\frac{1}{2}B_m^T & \bzero & \btau^{-1}_4
\end{bmatrix}, 
\end{split}
\end{align}
where $\btau_1, \btau_2, \btau_3$, and $\btau_4$ are properly chosen such that $\Phi$ is positive definite. 
Combining the composite of the reflected resolvents of $\Phi^{-1}\optA$ and $\Phi^{-1}\optB$ and the KM algorithm yields the iteration given as follows:
\begin{align}\label{eq:D-R}
\begin{split}
& \text{Calculate}\; J_{\Phi^{-1}\optA}: \supsub{\psi}{(k+1)}{} \coloneqq J_{\Phi^{-1}\optA}(\supsub{\tilde{\psi}}{(k)}{});\\
& \text{R-R updates}: \supsub{\hat{\psi}}{(k+1)}{} \coloneqq 2 \cdot \supsub{\psi}{(k+1)}{} - \supsub{\tilde{\psi}}{(k)}{}; \\
& \text{Calculate}\; J_{\Phi^{-1}\optB}: \supsub{\bar{\psi}}{(k+1)}{} \coloneqq J_{\Phi^{-1}\optB}(\supsub{\hat{\psi}}{(k+1)}{}); \\
& \text{K-M updates}: \supsub{\tilde{\psi}}{(k+1)}{} \coloneqq \supsub{\tilde{\psi}}{(k)}{} + 2\supsub{\gamma}{(k)}{}(\supsub{\bar{\psi}}{(k+1)}{} - \supsub{\psi}{(k+1)}{}).
\end{split}
\end{align}
Note that under Assumptions~\ref{asp:commtopo}, \ref{asp:obj-g} to \ref{asp:exist-vsgne}, and \ref{asp:convg}, we can prove a similar result concerning the quasinonexpansiveness and continuity of $R_{\Phi^{-1}\optA}$ and the convergence of the above iteration to a v-SGNE of the original problem as what we did in Theorem~\ref{thm:exact-convg-pf}. 

Likewise, by considering the unknown parameters $w^*$ in objectives and introducing parameter estimates $\hat{w}^{(k)}$ at each iteration $k$, we substitute the fixed-point iteration operator $\mathscr{R}_* = R_{\Phi^{-1}\optB} \circ R_{\Phi^{-1}\optA}$ with its estimate $\mathscr{R}^{(k)} = R_{\Phi^{-1}\optB} \circ R_{\Phi^{-1}\optA^{(k)}}$, and obtain the v-SGNE seeking dynamics described in Subroutine~\ref{alg:node-edge-g}. 
The proof of convergence in Section~\ref{sec:convg-analy} can then be carried over to verify that by combining Subroutine~\ref{alg:node-edge-g} and Subroutine~\ref{alg:param-est}, we can obtain learning dynamics that converge to a fixed point of the operator $\mathscr{R}_*$ almost surely. 

\begin{algorithm2e}
\SetAlgorithmName{Subroutine}{subroutine}{List of Subroutines}
\SetAlgoLined
\caption{Distributed Generalized Nash Equilibrium Seeking}\label{alg:node-edge-g}
\textbf{Available variables:} $\{{\tilde{y}}^{(k)}_{i}\}$, $\{{\tilde{\lambda}}^{(k)}_{i}\}$, $\{{\tilde{\mu}}^{(k)}_{ji}\}$, $\{\tilde{z}^{(k)}_{ji}\}$, $\{\hat{w}^{(k)}_i\}$ \;
\textbf{At the $k$-th iteration: }\\
\textbf{Each player $i \in \playerN$}: ($R_{\Phi^{-1}\optA}$)\\
\qquad Receive $\{{\tilde{y}}^{j(k)}_j\}_{j \in \neighbN{+}{i}}$ from its in-neighbors and $\{{\tilde{y}}^{i(k)}_j\}_{j \in \neighbN{-}{i}}$ from its out-neighbors\;
\qquad Receive $\{\tilde{\mu}^{(k)}_{ij}\}_{j \in \neighbN{-}{i}}$ and $\{\tilde{z}^{(k)}_{ij}\}_{j \in \neighbN{-}{i}}$ from its out-neighbors\;

\qquad $y^{j(k+1)}_{i} = \tilde{y}^{j(k)}_{i} - \tau_{1i}\Big(\frac{\rho}{2}(\tilde{y}^{j(k)}_{i} - \tilde{y}^{j(k)}_{j}) + \frac{1}{2}\tilde{\mu}^{(k)}_{ji}\Big), \quad \forall j \in \neighbN{+}{i}$\;

\qquad ${y}^{i(k+1)}_{i} = \underset{{y}^{i}_{i} \in \mathcal{X}_i}{\argmin}\Big\{ \mathbb{J}_i({y}^i_i; {y}^{+(k+1)}_i, \hat{w}^{(k)}_{i})  + \frac{1}{2\tau_{1i}}\norm{{y}^{i}_{i} - \tilde{y}^{i(k)}_{i}}^2_2 $ \\
\qquad\qquad\qquad\qquad $+ \frac{1}{2}(\rho\underset{j \in \neighbN{-}{i}}{\sum}(\tilde{y}^{i(k)}_{i} - \tilde{y}^{i(k)}_{j})+A_i^T\tilde{\lambda}^{(k)}_i - \sum_{j \in \neighbN{-}{i}}\tilde{\mu}^{(k)}_{ij})^T{y}^i_i\Big\}$\; 
\qquad $\lambda^{(k+1)}_i = \tilde{\lambda}^{(k)}_i + \tau_{2i}(A_iy^{i(k+1)}_{i} - \frac{1}{2}A_i\tilde{y}^{i(k)}_{i} - \frac{1}{2}(\sum_{j \in \neighbN{+}{i}}\tilde{z}^{(k)}_{ji} - \sum_{j \in \neighbN{-}{i}}\tilde{z}^{(k)}_{ij}) - c_i)$\; 
\qquad Receive $\{y^{j(k+1)}_{j}\}_{j \in \neighbN{-}{i}}$ and $\{\lambda^{(k+1)}_{j}\}_{j \in \neighbN{-}{i}}$ from its out-neighbors\;
\qquad $\mu^{(k+1)}_{ji} = \tilde{\mu}^{(k)}_{ji} + \tau_{3i}((y^{j(k+1)}_{i} - y^{j(k+1)}_{j}) - \frac{1}{2}(\tilde{y}^{j(k)}_{i} - \tilde{y}^{j(k)}_{j}))$\;
\qquad $z^{(k+1)}_{ji} = \tilde{z}^{(k)}_{ji} + \tau_{4i}((\lambda^{j(k+1)}_{i} - \lambda^{j(k+1)}_{j}) - \frac{1}{2}(\tilde{\lambda}^{(k)}_{i} - \tilde{\lambda}^{(k)}_{j}))$\; 
\qquad \textbf{R-R updates: } $\hat{\psi}^{(k+1)}_i = 2\psi^{(k+1)}_i - \tilde{\psi}^{(k)}_i$\;
\textbf{Each player $i \in \playerN$}: ($R_{\Phi^{-1}\optB}$ and K-M updates)\\
\qquad Receive $\{{\hat{y}}^{j(k+1)}_j\}_{j \in \neighbN{+}{i}}$ from its in-neighbors and $\{{\hat{y}}^{i(k+1)}_j\}_{j \in \neighbN{-}{i}}$ from its out-neighbors\;
\qquad Receive $\{\hat{\mu}^{(k+1)}_{ij}\}_{j \in \neighbN{-}{i}}$ and $\{\hat{z}^{(k+1)}_{ij}\}_{j \in \neighbN{-}{i}}$ from its out-neighbors\;

\qquad $\check{y}^{j(k+1)}_{i} = \hat{y}^{j(k+1)}_{i} - \tau_{1i}\Big(\frac{\rho}{2}(\hat{y}^{j(k+1)}_{i} - \hat{y}^{j(k+1)}_{j}) + \frac{1}{2}\hat{\mu}^{(k+1)}_{ji}\Big), \quad \forall j \in \neighbN{+}{i}$\;

\qquad $\check{y}^{i(k+1)}_{i} = \hat{y}^{i(k+1)}_{i} - \tau_{1i}\Big(\frac{\rho}{2}(\underset{{j\in \neighbN{-}{i}}}{\sum}(\hat{y}^{i(k+1)}_{i} - \hat{y}^{i(k+1)}_{j})) + \frac{1}{2}A_i^T\hat{\lambda}^{(k+1)}_{i} -\underset{{j\in \neighbN{-}{i}}}{\sum} \frac{1}{2}\hat{\mu}^{(k+1)}_{ij}\Big)$\;

\qquad Obtain $\bar{y}^{i(k+1)}_{i}$ and $\bar{y}^{j(k+1)}_{i}$ with $j \in \neighbN{+}{i}$ by solving the following problem: 
\begin{equation*}
\begin{cases}
\minimize_{y_i} & \frac{1}{2}\norm{y^i_i - \check{y}^{i(k+1)}_{i}}^2_2 + \frac{1}{2}\sum_{j\in \neighbN{+}{i}} \norm{y^j_i - \check{y}^{j(k+1)}_{i}}^2_2\\
\subj & h(y^i_i, \{y^j_i\}_{j \in \neighbN{+}{i}}) \leq \bzero \\
\end{cases}
\end{equation*}

\qquad $\bar{\lambda}^{(k+1)}_{i} = \proj_{\rset{m}{+}}\big[ \hat{\lambda}^{(k+1)}_i + \tau_{2i}(A_i \bar{y}^{i(k+1)}_i - \frac{1}{2}A_i\hat{y}^{i(k+1)}_i) - \frac{1}{2}(\underset{j \in \neighbN{+}{i}}{\sum}\hat{z}^{(k+1)}_{ji} - \underset{j \in \neighbN{-}{i}}{\sum}\hat{z}^{(k+1)}_{ij}) \big]$\;
\qquad Receive $\{\bar{y}^{j(k+1)}_{j}\}_{j \in \neighbN{-}{i}}$ and $\{\bar{\lambda}^{(k+1)}_{j}\}_{j \in \neighbN{-}{i}}$ from its out-neighbors\;
\qquad $\bar{\mu}^{(k+1)}_{ji} = \hat{\mu}^{(k+1)}_{ji} + \tau_{3i}((\bar{y}^{j(k+1)}_{i} - \bar{y}^{j(k+1)}_{j}) - \frac{1}{2}(\hat{y}^{j(k+1)}_{i} - \hat{y}^{j(k+1)}_{j}))$\;
\qquad $\bar{z}^{(k+1)}_{ji} = \hat{z}^{(k+1)}_{ji} + \tau_{4i}((\bar{\lambda}^{j(k+1)}_{i} - \bar{\lambda}^{j(k+1)}_{j}) - \frac{1}{2}(\hat{\lambda}^{(k+1)}_{i} - \hat{\lambda}^{(k+1)}_{j}))$\; 
\qquad \textbf{K-M updates: } $\tilde{\psi}^{(k+1)}_i = \tilde{\psi}^{(k)}_i + 2\gamma^{(k)}(\bar{\psi}^{(k+1)}_i - \psi^{(k+1)}_i)$.

\textbf{Each player $i$ returns the solution of the augmented best-response subproblem: $y^{i(k+1)}_i$}
\end{algorithm2e}

\subsection{Optimize Expected-Value Objectives Using Inexact Solvers}
In Subroutine~\ref{alg:node-edge} and Subroutine~\ref{alg:node-edge-g}, when updating its decision, each player $i$ should solve an augmented best-response subproblems w.r.t. $\hat{\mathbb{J}}_i(\cdot; \hat{w}^{(k)})$, the details of which are given as:
\begin{align}\label{eq:argmin-subprob}
\begin{split}
& \text{PPA for SNEP}: \\
& \qquad \tilde{y}^{i(k+1)}_{i} = \underset{\tilde{y}^{i}_{i} \in \mathcal{X}_i}{\argmin}\Big\{ \mathbb{J}_i(\tilde{y}^i_i; \tilde{y}^{+(k+1)}_i, \hat{w}^{(k)}_{i}) + \rho(\underset{j \in \neighbN{-}{i}}{\sum}{y}^{i(k)}_{i} - {y}^{i(k)}_{j})^T\tilde{y}^i_i + \frac{1}{2\tau_{i0}}\norm{\tilde{y}^{i}_{i} - {y}^{i(k)}_{i}}^2_2 \Big\}; \\
& \text{DR for SGNEP}: \\
& \qquad {y}^{i(k+1)}_{i} = \underset{{y}^{i}_{i} \in \mathcal{X}_i}{\argmin}\Big\{ \mathbb{J}_i({y}^i_i; {y}^{+(k+1)}_i, \hat{w}^{(k)}_{i}) + \frac{1}{2\tau_{i}}\norm{{y}^{i}_{i} - \tilde{y}^{i(k)}_{i}}^2_2 \\
& \qquad\qquad\qquad\qquad\qquad  + \frac{1}{2}(\rho\underset{j \in \neighbN{-}{i}}{\sum}(\tilde{y}^{i(k)}_{i} - \tilde{y}^{i(k)}_{j})+A_i^T\tilde{\lambda}^{(k)}_i - \sum_{j \in \neighbN{-}{i}}\tilde{\mu}^{(k)}_{ij})^T{y}^i_i\Big\}. 
\end{split}
\end{align}
A fundamental question that arises here is how we can solve these problems when the analytical expression of the expected-value objective $\hat{\mathbb{J}}_i$ is complex or even does not exist. 
One potential answer could be leveraging some stochastic optimization methods, such as projected stochastic subgradient descent, random coordinate descent, etc. 
Then the ensuring question would be how we can adjust these stochastic optimization methods and properly tune their parameters to fit them into the proposed learning dynamics and ensure the convergence of the generated sequence. 
In this appendix, we will leverage projected stochastic subgradient descent to construct an inexact solver and provide explicit statements concerning the choices of step sizes, which are largely based on the results in \citeApndx[Sec.~IV.B]{huang2021sgnep}. 
Since the analysis of one can be straightforwardly carried over to the other, in the remaining content of this appendix, we will focus on dissecting the PPA-based dynamics for SNEP. 

Previously, we use $\mathscr{R}^{(k)}$ to denote an estimate for the exact fixed-point iteration operator $\mathscr{R}_*$, with the unknown parameters $w^*$ replaced by the current estimate $\hat{w}^{(k)}$. 
Here, we let ${\mathscr{R}}^{(k)}_s$ denote the (scenario-based) approximate operator to the estimate $\mathscr{R}^{(k)}$. 
With the introduction of the approximation to the involved subproblems, the overall error can be defined as:
\begin{align*}
\begin{split}
\epsilon^{(k)} \coloneqq {\mathscr{R}}^{(k)}_s(y^{(k)}) - \mathscr{R}_*(y^{(k)}), 
\; \text{and} \; 
\varepsilon^{(k)} \coloneqq \norm{\epsilon^{(k)}}_\pspace. 
\end{split}
\end{align*}
According to the conclusion of Theorem~\ref{thm:main-convg-thm}, it suffices for us to ensure that $(\norm{y^{(k)}}_\pspace)_{k \in \nset{}{}}$ is bounded a.s. and $\sum_{k \in \nset{}{}} \gamma^{(k)}\expt{}{\varepsilon^{(k)} \mid \mathcal{F}_{k}} < +\infty$ a.s. 
Furthermore, the error norm $\varepsilon^{(k)}$ can be further decomposed as $\varepsilon^{(k)} \leq \norm{{\mathscr{R}}^{(k)}_s(y^{(k)}) - \mathscr{R}^{(k)}(y^{(k)})}_\pspace + \norm{\mathscr{R}^{(k)}(y^{(k)}) - \mathscr{R}_*(y^{(k)})}_\pspace$. 
We let $\varepsilon^{(k)}_s \coloneqq \norm{{\mathscr{R}}^{(k)}_s(y^{(k)}) - \mathscr{R}^{(k)}(y^{(k)})}_\pspace$ and $\varepsilon^{(k)}_e \coloneqq \norm{\mathscr{R}^{(k)}(y^{(k)}) - \mathscr{R}_*(y^{(k)})}_\pspace$. 
Lemma~\ref{le:resol-upper-bd} implies that $\expt{}{\varepsilon^{(k)}_e \mid \mathcal{F}_k} \leq (\tilde{\alpha}\norm{y^{(k)}}_\pspace + \tilde{\beta}) \expt{}{\norm{\Delta\hat{w}^{(k)}}_2 \mid \mathcal{F}_k}$.

To derive a decaying upper bound for $\varepsilon^{(k)}_s$, we first define the following augmented scenario-based objective function for each player $i$ at iteration $k$:
\begin{align*}
\begin{split}
\hat{J}^{(k)}_i(y^i_i; \hat{w}^{(k)}, \xi^{(k)}_{i,t}) \coloneqq & J_i(y^i_i; s_i(\tilde{y}^{+(k+1)}_i; \xi^{(k)}_{i,t}, \hat{w}^{(k)})) \\
& + \rho(\underset{j \in \neighbN{-}{i}}{\sum}{y}^{i(k)}_{i} - {y}^{i(k)}_{j})^T\tilde{y}^i_i + \frac{1}{2\tau_{i0}}\norm{\tilde{y}^{i}_{i} - {y}^{i(k)}_{i}}^2_2.
\end{split}
\end{align*}
We will use $k$ to index the major iterations (the iteration of the SNE seeking Subroutine~\ref{alg:node-edge}) and $t$ to index the minor iterations (the iteration of the inexact solver in $\argmin$ subproblems of Subroutine~\ref{alg:node-edge}). 
Let $T^{(k)}_i$ denote the total number of the projected stochastic subgradient steps taken in the $k$-th major iteration by player $i$.
The subgradient of the scenario-based objective function at the $k$-th major iteration and the $t$-th minor iteration is denoted by $g^{(k)}_{i,t} \in \partial_{y^i_i}\hat{J}^{(k)}_i(\tilde{y}^{i(k+1)}_{i,t}; \hat{w}^{(k)}, \xi^{(k)}_{i,t})$, where $t = 0, 1, \ldots, T^{(k)}_i-1$. 
The following assumptions are made concerning the subgradient:
\begin{assumption}\label{asp:proj-stoch-subgrad}
For each player $i \in \playerN$, at each major iteration $k$ and minor iteration $t$ of Subroutine~\ref{alg:node-edge}, there exists a $g^{(k)}_{i,t} \in \partial_{y^i_i}\hat{J}^{(k)}_i(\tilde{y}^{i(k+1)}_{i,t}; \hat{w}^{(k)}, \xi^{(k)}_{i,t})$ such that the following two statements hold:
\begin{outline}[enumerate]
\1 (Unbiasedness) $\expt{}{g^{(k)}_{i,t} \mid \sigma\{\mathcal{F}_k, \xi^{(k)}_{i, [t]}\}}$ is almost surely a subgradient of the expected-value augmented objective $\hat{\mathbb{J}}_i^{(k)}(\cdot)$ at $y^{i(k+1)}_{i,t}$, where $\xi^{(k)}_{i, [t]} \coloneqq \{ \xi^{(k)}_{i, 0}, \ldots, \xi^{(k)}_{i,t-1} \}$ with $\xi^{(k)}_{i, [0]} \coloneqq \varnothing$;
\1 (Finite variance) $\expt{}{\norm{g^{(k)}_{i,t}}^2_2 \mid \mathcal{F}_k} \leq \alpha_{g,i}^2\norm{\tilde{\psi}^{(k)}}^2_2 + \beta_{g,i}^2$ a.s. for some positive constants $\alpha_{g,i}$ and $\beta_{g,i}$.
\end{outline}
\end{assumption}
The proposed projected stochastic subgradient solver for $\argmin$ subproblems of Subroutine~\ref{alg:node-edge} is given in Algorithm~\ref{alg:proj-stoch-subgrad}. 
\begin{algorithm2e}
\SetAlgoLined
\caption{Projected Stochastic Subgradient Inexact Solver}
\label{alg:proj-stoch-subgrad}
\textbf{For each player $i \in \playerN$, at the $k$-th major iteration of Subroutine~\ref{alg:node-edge}:} \\
\textbf{Initialize:} $\tilde{y}^{i(k+1)}_{i,0} \coloneqq {y}^{i(k)}_{i}; $\\
\For{$t = 0 \text{ to } T^{(k)}_i - 1$}{
$\tilde{y}^{i(k+1)}_{i,t+1} \coloneqq \proj_{\mathcal{X}_i}[\tilde{y}^{i(k+1)}_{i,t} - \kappa_{i,t} \cdot g^{(k)}_{i,t}]$, 
with $\kappa_{i,t} \coloneqq \frac{2\tau_{i0}}{t+2}$;\\
}
\textbf{Return:} $\tilde{y}^{i(k+1)}_{i} \coloneqq \tilde{y}^{i(k+1)}_{i, T^{(k)}_{i}}$.
\end{algorithm2e}
\begin{lemma}\label{le:proj-stoch-convg-rate}
Suppose Assumptions~\ref{asp:commtopo} to \ref{asp:convg} and \ref{asp:proj-stoch-subgrad} hold. 
If the $\argmin$ subproblems in Subroutine~\ref{alg:node-edge} are solved using the inexact solver in Algorithm~\ref{alg:proj-stoch-subgrad}, then
\begin{align*}
\expt{}{\varepsilon^{(k)}_s \mid \mathcal{F}_k} \leq 2\tau_{i0}\sqrt{T^{(k)}_{i}}(\alpha_{g,i}\norm{y^{(k)}}_2 + \beta_{g,i}) \; \text{a.s.}
\end{align*}
\end{lemma}
\begin{proof}
See \citeApndx[Appendix~B]{huang2021sgnep}.
\end{proof}
By choosing $T^{(k)}_i \propto k^{\alpha_3}$ with $\alpha_3 \geq 1$ and $\gamma^{(k)} = 1/k^{\alpha_1}$ with $1/2 < \alpha_1 \leq 1$, we can similarly establish the almost-sure convergence of the learning dynamics with inexact solver as in the proof of Theorem~\ref{thm:bdd-est-seq-sumb}.

\subsection{Case Study and Numerical Experiments}

\subsubsection{Game Setup}
In this section, we evaluate the performance of the proposed dynamics with a variant of the Nash-Cournot game over a network. 
In this variant, $N$ manufacturers/players indexed by $\mathcal{N}$ are involved in producing several homogeneous commodities and competing for different local markets. 
For each manufacturer $i$ in this network, its scenario-based local objective function is given by $J_i(x_i; s_i(x^+_i; \xi_i, w^*_i)) \coloneqq f_i(x_i) - (b_0 + b^*_{i} - \sum_{j \in \neighbN{+}{i} \cup \{i\}}P^*_{ji}H^{*}_{j}x_j + \xi_i) \cdot H^{*}_ix_i$, where $x_i \in \mathcal{X}_i \subseteq \rset{n_i}{}$ denotes the decision vector of player $i$, $f_i(x_i) = x_i^TQ_ix_i + q_i^Tx_i$ denotes the local production cost function with $Q_i \in \pdset{n_i}{++}$ and $q_i \in \rset{n_i}{+}$, the matrix $H^*_i \in \rset{1 \times n_i}{+}$, and $P^*_{ji}$ corresponds to the certain entry in the weighted adjacency matrix $P^*$, and the unknown parameters $w^*_i$ is defined as $w^*_i \coloneqq [b^*_{i}; [P^*_{ji}H^{*T}_j]_{j \in \neighbN{+}{i}}]$. 
The local feasible set $\mathcal{X}_i$ is set to be the direct product of $n_i$ connected compact intervals $[0, \mathcal{X}_{ij\text{max}}]$ for $j = 1, \ldots, n_i$. 
Overall, each player $i \in \playerN$, given the decision vectors of its neighbors $x^+_i$ aims to solve the following optimization problem:
\begin{align}
\minimize_{x_i \in \mathcal{X}_i} \expt{}{J_i(x_i; s_i(x^+_i; \xi_i, w^*_i))}.
\end{align}

For notational brevity, let $n \coloneqq \sum_{i \in \playerN} n_i$, $x \coloneqq [x_1; \cdots; x_N]$, $H^* \coloneqq \blkd{H^*_1, \ldots, H^*_N} \in \rset{N \times n}{}$, $Q \coloneqq \blkd{Q_1, \ldots, Q_N}$, $D^*$ denote the weighted degree matrix associated with $P^*$. 
In this case, the pseudo-gradient operator $\gjacob_{w^*}$ is given by:
\begin{align}
\gjacob_{w^*}: x \mapsto M_Fx + [q_i]_{i \in \playerN} - [H^{*T}_i(b_0+b^*_i)]_{i \in \playerN},
\end{align}
where $M_F \coloneqq 2Q + H^{*T}(D^*+P^*)H^*$. 
We can choose proper parameters such that $\frac{1}{2}(M_F + M_F^T)$ is positive definite. 
The minimal eigenvalue $\ubar{\sigma}_{\gjacob_{w^*}} > 0$ of $\frac{1}{2}(M_F + M_F^T)$ and the maximal singular value $\bar{s}_{\gjacob_{w^*}} > 0$ of $M_F$ are the strongly monotone constant $\eta$ and the Lipschitz constant $\theta_1$ of the pseudogradient $\gjacob$, respectively. 
Similarly, the extended pseudogradient $\extgjacob_{w^*}$ can be expressed as: 
\begin{align}
\extgjacob_{w^*}: y \mapsto \mathcal{R}(I_N \otimes M_F)y + [q_i]_{i \in \playerN} - [H^{*T}_i(b_0+b^*_i)]_{i \in \playerN}.
\end{align}
The Lipschitz constant of $\extgjacob_{w^*}$ is given by the greatest singular value $\bar{s}_{\extgjacob_{w^*}}$ of $\mathcal{R}(I_N \otimes M_F)$. 

\subsubsection{Simulation Results}

We consider a game played by $N = 10$ players. 
The communication graph consists of an undirected circle and $10$ randomly selected edges. 
Each local feasible set $\mathcal{X}_i$ is the direct product of $n_i$ intervals, i.e., $\mathcal{X}_i \coloneqq \prod_{j=1}^{n_i}[0, \mathcal{X}_{ij \max}]$. 
The related parameters are drawn uniformly from suitable intervals. 
The dimension $n_i$ satisfies $n_i \sim U\{3, 4, 5\}$, and each upper bound of the local feasible sets satisfies $\mathcal{X}_{ij \max} \sim U[10, 20]$. 
We set the base price to be $b_0 = 50$ and $b^*_i \sim U[3, 7]$. 
The production cost $f_i$ has each entry of $Q_i$ chosen from $U[4.4, 4.6]$ and each entry of $q_i$ chosen from $U[1, 1.2]$. 
The entries of adjacency matrix $P^*$ are uniformly drawn from $U[0.8, 1]$ and later regularized to have each row summing to one. 
Each entry of the unknown weights $H^*_i$ is selected from $U[0.8, 4.5]$.
Each random noise $\xi_i$ has the truncated Gaussian distribution with mean zero and standard deviation $0.5$. 
Player $i$ chooses each entry of its random exploration vector from $U[-\delta_i, +\delta_i]$, where $\delta_i = \frac{0.01}{2\sqrt{n_i}}\min\{\mathcal{X}_{ij\max}: j = 1, \ldots, n_i\}$. 
The parameters $\rho$ and $\btau$ are properly chosen such that $\optT$ is maximally monotone and $\Phi$ is positive definite. 

We set the sequence $(\gamma^{(k)})_{k\in \nset{}{}}$ to be $\gamma^{(k)} = k^{-0.501}$. 
The package CVXPY 1.0.31 (\citeApndx{diamond2016cvxpy, agrawal2018rewriting}) is used to construct the exact solver for Subroutine~\ref{alg:node-edge} and the quadratic programming inside Subroutine~\ref{alg:param-est}. 
Since the formulations of the expected-value objectives in this example are immediate, we first consider the following two cases that involved the exact solver: (i) SNE seeking with the precise parameters \eqref{eq:ppa-km}; 
(ii) SNE learning with the estimated parameters (Algorithm~\ref{alg:assemble}). 
The performances of the proposed algorithm are illustrated in Fig.~\ref{fig:nash-cournot-exact}. 
We then investigate the above two cases with the inexact solver described in Algorithm~\ref{alg:proj-stoch-subgrad} with $T^{(k)}_i \coloneqq \ceil{0.01k} + 10$. 
The performances of this substitution are illustrated in Fig.~\ref{fig:nash-cournot-inexact}. 
The first 150 points of Fig.~\ref{fig:nash-cournot-exact} (b) and Fig.~\ref{fig:nash-cournot-inexact} (b) are discarded to better illustrate the trends throughout the iterations. 
In both cases, we see that the normalized distances to the true SNE decrease quickly in the first $2 \times 10^4$ iterations, although in the experiments using the inexact solver, the relative lengths of the updating step fluctuate dramatically. 
Nevertheless, the convergence speeds to the true SNE slow down as the iterations proceed, mainly dominated by the convergence rates of the estimated parameters to their true correspondences. 
It remains an interesting question to rigorously derive an upper bound for the convergence rate of Algorithm~\ref{alg:assemble}, and investigate its dependence on the dimensions of unknown parameters as well as those of decision vectors.

\begin{figure}
    \centering
    \includegraphics[width=\textwidth]{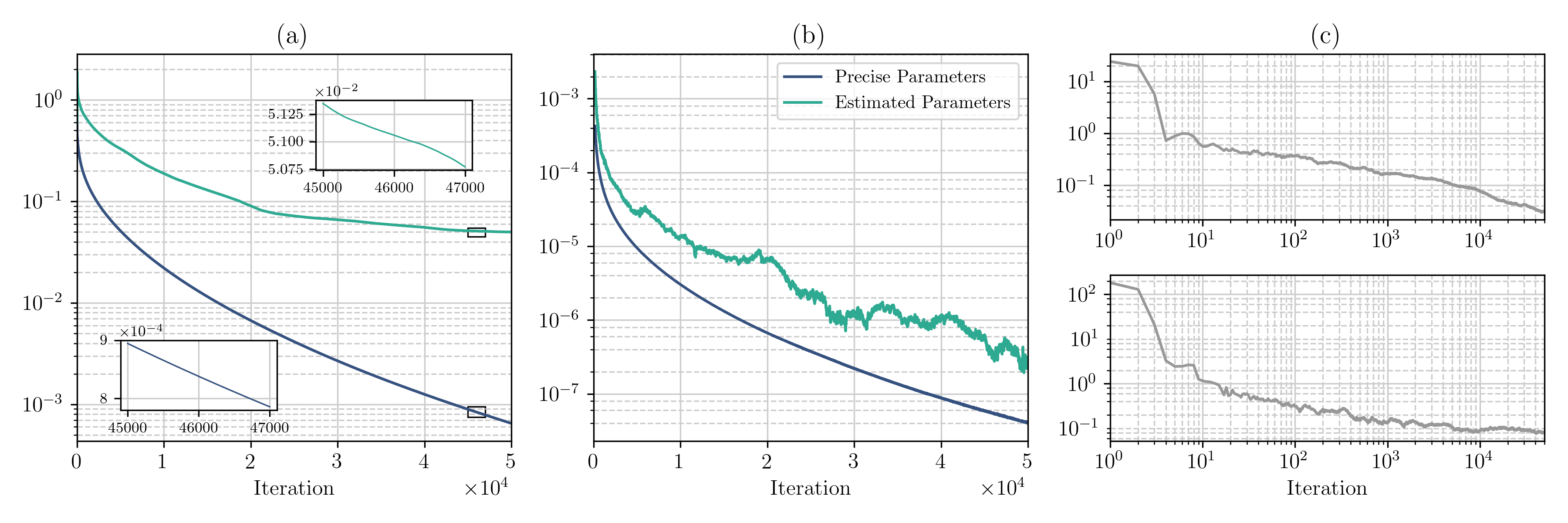}
    \caption{Performances of Algorithm~\ref{alg:assemble} with Exact Solver.
    (a) The normalized distances to the true SNE, i.e., $\frac{1}{N}\sum_{i \in \playerN} \norm{y^{i(k)}_{i} - y^{i*}_i}_2/\norm{y^{i*}_i}_2$. (b) The relative lengths of the updating step at each iteration, i.e., $\frac{1}{N}\sum_{i \in \playerN} \norm{y^{i(k+1)}_{i} - y^{i(k)}_i}_2/\norm{y^{i(k)}_i}_2$. (c Top) The normalized distances ofthe estimated weights to the precise weights, i.e., $\frac{1}{N}\sum_{i \in \playerN} (\sum_{j \in \neighbN{+}{i}}\norm{\hat{P}^{(k)}_{ji}\hat{H}^{(k)}_{j} - P^*_{ji}H^{*}_{j}}^2_2)^{1/2}/(\sum_{j \in \neighbN{+}{i}}\norm{P^*_{ji}H^{*}_{j}}^2_2)^{1/2}$. (c Bottom) The normalized distances of the estimated biases to the precise biases, i.e., $\frac{1}{N}\sum_{i \in \playerN}\abs{b^{(k)}_i - b^*_i}/\abs{b^*_i}$. 
    }
    \label{fig:nash-cournot-exact}
\end{figure}

\begin{figure}
    \centering
    \includegraphics[width=\textwidth]{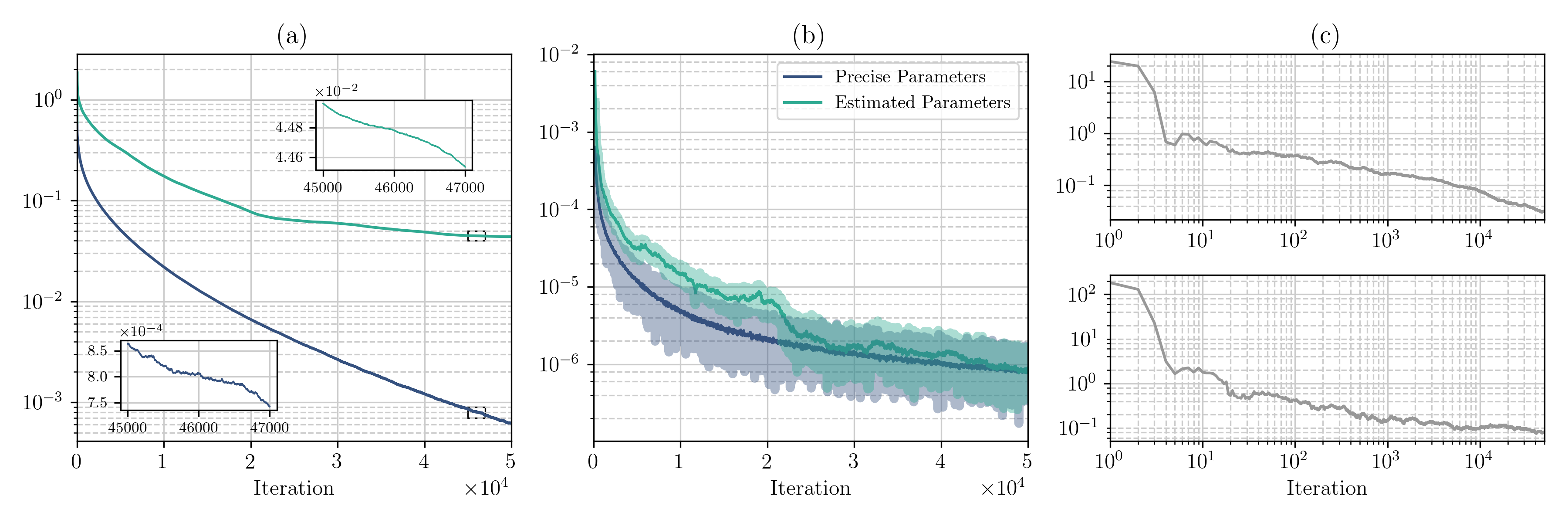}
    \caption{Performances of Algorithm~\ref{alg:assemble} with Inexact Solver. (a) The normalized distances to the true SNE, i.e., $\frac{1}{N}\sum_{i \in \playerN} \norm{y^{i(k)}_{i} - y^{i*}_i}_2/\norm{y^{i*}_i}_2$. (b) The relative lengths of the updating step at each iteration, i.e., $\frac{1}{N}\sum_{i \in \playerN} \norm{y^{i(k+1)}_{i} - y^{i(k)}_i}_2/\norm{y^{i(k)}_i}_2$. 
    The thick and semi-transparent line denotes the real fluctuations of the metric, while the solid and thin line displays the simple moving averages of the metric with a window size of 80.
    (c Top) The normalized distances of the estimated weights to the precise weights, i.e., $\frac{1}{N}\sum_{i \in \playerN} (\sum_{j \in \neighbN{+}{i}}\norm{\hat{P}^{(k)}_{ji}\hat{H}^{(k)}_{j} - P^*_{ji}H^{*}_{j}}^2_2)^{1/2}/(\sum_{j \in \neighbN{+}{i}}\norm{P^*_{ji}H^{*}_{j}}^2_2)^{1/2}$. 
    (c Bottom) The normalized distances of the estimated biases to the precise biases, i.e., $\frac{1}{N}\sum_{i \in \playerN}\abs{b^{(k)}_i - b^*_i}/\abs{b^*_i}$. 
    }
    \label{fig:nash-cournot-inexact}
\end{figure}

\subsection{Other Motivating Examples of Locally Coupled Network Games}
\begin{example}(Cognitive radio systems)
Consider the problem of designing a cognitive radio system whose transmissions are over single-input single-output (SISO) frequency-selective channels as discussed in \citeApndx{pang2010design, scutari2012monotone}. 
It is composed of $N^p$ primary users (PUs), $N$ secondary users (SUs), and $n^s$ available subcarriers. 
These $N$ SUs, indexed by $\playerN \coloneqq \{1, \ldots, N\}$, participate in the frequency resource allocation game, where each SU $i$ competes against each other to maximize its own information rate by determining its power allocation vector $x_i \in \rset{n^s_i}{}$ over the $n^s_i$ subcarriers available to it. 
Here, the objective function of each SU $i \in \playerN$ can often be expressed as:
\begin{equation}\label{eq:cr-objt}
\begin{split}
& \mathbb{J}_i(x_{i}; x^+_{i}, w^*_i) \coloneqq 
\bexpt{\xi_i}{-\sum_{s \in \mathcal{S}_i} J_{i,s}(x_i; x^+_i, w^*_i, \xi_i(s))}, \text{where} \\
& J_{i,s}(x_i; x^+_i, w^*_i, \xi_i(s)) \coloneqq \log(1 + \frac{\abs{H_{ii}(s)}^2 [x_{i}]_s}{\sigma_i^2(s) + \sum_{s \in \mathcal{S}_j, j \neq i} \abs{H_{ij}(s)}^2[x_j]_s + \xi_i(s) })
\end{split}
\end{equation}
where $\mathcal{S}_i \coloneqq \{1, \ldots, n^s_i\}$ is the index set for the subcarriers accessible to SU $i$, $\sigma^2_i(s)$ denotes the thermal noise power over the subcarrier $s$, $H_{ij}(s)$ is the channel transfer function between the secondary transmitter $j$ and the receiver $i$, $[x_i]_s$ represents the entry of vector $x_i$ associated with the power allocation decision of subcarrier $s$, and $\xi_i(s)$ represents uncertainty in the realization of the objective. 
In this network game, the communication topology is constructed in a way such that each SU $i$ can communicate with the SUs who have the direct competing interest in occupying a certain subcarrier $s \in \mathcal{S}_i$, i.e., $\mathcal{N}^+_i \coloneqq \{j \in \playerN \mid (j, i) \in \mathcal{E}\}$ satisfies $\mathcal{N}^+_i = \{j \in \playerN \mid \exists s \in \mathcal{S}_i \text{ s.t. } s \in \mathcal{S}_j, j \neq i\}$.
Assume that $\abs{H_{ii}(s)}^2$ is available to each SU $i$, while $\sigma_i^2(s)$ and $\{\abs{H_{ij}(s)}^2\}_{j \in \neighbN{}{-i}}$ are the unknown parameters to be learnt, i.e., $w^*_i \coloneqq [\sigma_i^2(s); [H_{ij}(s)]_{j \in \neighbN{}{-i}}]_{s \in \mathcal{S}_i}$. 
The associated feasible region $\mathcal{W}_i$ is assumed to be some non-negative box set. 

Furthermore, when each SU $i$ tries to minimize the above objective, it should fulfill some local and global constraints such that it works within some prescribed transmit power and induces tolerable degradation on PU's performance. 
We use $Q_{pi}(s)$ to represent the channel transfer function between the secondary transmitter $i$ and the primary receiver $p$ over the subcarrier $s$, and $I^{\text{tot}}_{pi}$ the maximum interferences allowed to be generated by the SU $i$ at the primary receiver $p$ over the whole spectrum.  
The local feasible set of each SU $i$ is defined as follows:
\begin{equation}\label{eq:cr-local-cnstr}
\begin{split}
& \mathcal{X}_i \coloneqq \{x_i \in \rset{n^s_i}{}\mid \bzero \leq x_i\leq b_i, \bone_{n^s_i}^Tx_i \leq \bar{\kappa} \cdot \bone_{n^s_i}^Tb_i, \\
& \qquad {\textstyle\sum}_{s \in \mathcal{S}_i}\abs{Q_{pi}(s)}^2[x_i]_s \leq I^{\text{tot}}_{pi},\;\forall p\}, 
\end{split}
\end{equation}
where $b_i \in \rset{n^s}{}$ is a vector with all entries positive, and $\bar{\kappa} \in (0, 1)$ is some prescribed constant. 
In addition to the local constraints in \eqref{eq:cr-local-cnstr}, it is possible that under some circumstances, the power allocation vector $\{x_i\}$ of all SUs should collectively satisfy the following global constraints for $p = 1, \ldots, N^p$:
\begin{equation}\label{eq:cr-glb-cnstr}
\begin{split}
\sum_{i\in\playerN, s \in \mathcal{S}_i} \abs{Q_{pi}(s)}^2[x_i]_s \leq I^{\text{peak}}_p(s), \forall s = 1, \ldots, n^s, \\
\sum_{i\in\playerN}\sum_{s \in \mathcal{S}_i} \abs{Q_{pi}(s)}^2[x_i]_s \leq I^{\text{tot}}_p,
\end{split}
\end{equation}
where $I^{\text{peak}}_p(s)$ and $I^{\text{tot}}_p$ represent the maximum interferences allowed to be generated by all SUs at the primary receiver $p$ over the subcarrier $k$ and the whole spectrum, respectively. 
To satisfy Assumption~\ref{asp:olse} (ii), suppose at each iteration $k \in \nset{}{}$, each SU $i$ will receive the realized values $[{J_{i,s}]_{s \in \mathcal{S}_i}}$ of all the accessible subcarriers, and Subroutine~\ref{alg:param-est} is carried out to estimate the unknown parameters for each subcarrier separately. 
For the fulfillment of other assumptions, we refer the interested readers to \citeApndx{pang2010design, scutari2012monotone} for a more detailed discussion. 
\end{example}

\bibliographystyleApndx{plainnat} 
\bibliographyApndx{referappdx}
\end{document}